\documentclass[journal]{IEEEtran}

\usepackage{cite}
\ifCLASSINFOpdf
  \usepackage[pdftex]{graphicx}
\else
  \usepackage[dvips]{graphicx}
\fi
\usepackage{epstopdf}
\usepackage{amssymb}
\usepackage{amsthm}
\usepackage{amsmath,mathtools}
\usepackage{mathrsfs}
\usepackage{algorithmicx}
\usepackage{algpseudocode}
\usepackage{float}
\usepackage{hyperref}
\usepackage{color}
\usepackage{braket}
\usepackage{xifthen}
\usepackage{needspace}
\usepackage{enumitem}

\newtheorem{theorem}{Theorem}
\newtheorem{remark}{Remark}

\newtheorem{definition}{Definition}
\newtheorem{proposition}{Proposition}
\newtheorem{corollary}{Corollary}

\newcommand\numberthis{\addtocounter{equation}{1}\tag{\theequation}}

\newenvironment{proof-sketch}{\proof}{\endproof}

\newcommand{\mb}{\mathbf}
\newcommand{\mc}{\mathcal}
\newcommand{\agent}{_{\mc A}}
\newcommand{\cloud}{_{\mc C}}
\newcommand{\target}{_{\mc T}}

\DeclareMathOperator*{\argmin}{arg\,min}
\DeclareMathOperator*{\argmax}{arg\,max}

\newcounter{protocol}
\newenvironment{protocol}[1][]%
  {
     \needspace{2\baselineskip}
    \noindent \rule{\linewidth}{1pt} \endgraf
    \refstepcounter{protocol}
    \centering \textsc{Protocol}~\theprotocol%
    \ifthenelse{\isempty{#1}}{}{:\ #1}
  }{
  \vspace{-5pt}
  \noindent \rule{\linewidth}{1pt}\vspace{2pt}
  }

\begin{document}
\title{Cloud-based Quadratic Optimization with \\Partially Homomorphic Encryption}

\author{Andreea~B.~Alexandru,~\IEEEmembership{Student Member,~IEEE,}
        Konstantinos~Gatsis,~\IEEEmembership{Member,~IEEE,}
        Yasser~Shoukry,~\IEEEmembership{Member,~IEEE,}
  	Sanjit~A.~Seshia,~\IEEEmembership{Fellow,~IEEE,}
        Paulo~Tabuada,~\IEEEmembership{Fellow,~IEEE,}
        and~George~J.~Pappas,~\IEEEmembership{Fellow,~IEEE}
\thanks{A.~B. Alexandru, K.~Gatsis, G.~J. Pappas are with the Department
of Electrical and Systems Engineering, University of Pennsylvania, Philadelphia,
PA 19104, USA, e-mail: aandreea,kgatsis,pappasg@seas.upenn.edu.}
\thanks{Y. Shoukry is with the Deptartment of Electrical Engineering and Computer Science, University of California, Irvine, CA 92697, USA; e-mail: yshoukry@uci.edu.}
\thanks{S.~A. Seshia is with the Department of Electrical Engineering and Computer Sciences, University of California, Berkeley, CA 94720, USA, e-mail: sseshia@eecs.berkeley.edu.}
\thanks{P. Tabuada is with the Department of Electrical Engineering, University of California, Los Angeles,
 Los Angeles, CA 90095, USA, e-mail: tabuada@ee.ucla.edu.}
}

\maketitle

\begin{abstract}
The development of large-scale distributed control systems has led to the outsourcing of costly computations to cloud-computing platforms, as well as to concerns about privacy of the collected sensitive data.
This paper develops a \mbox{cloud-based} protocol for a quadratic optimization problem involving multiple parties, each holding information it seeks to maintain private. The protocol is based on the projected gradient ascent on the Lagrange dual problem and exploits partially homomorphic encryption and secure multi-party computation techniques. Using formal cryptographic definitions of indistinguishability, the protocol is shown to achieve computational privacy, i.e., there is no computationally efficient algorithm that any involved party can employ to obtain private information beyond what can be inferred from the party's inputs and outputs only. In order to reduce the communication complexity of the proposed protocol, we introduced a variant that achieves this objective at the expense of weaker privacy guarantees. We discuss in detail the computational and communication complexity properties of both algorithms theoretically and also through implementations. We conclude the paper with a discussion on computational privacy and other notions of privacy such as the non-unique retrieval of the private information from the protocol outputs.
\end{abstract}

\IEEEpeerreviewmaketitle

\section{Introduction}\label{sec:introduction}

The recent push towards increasing connectivity of devices is enabling control applications to span wider geographical areas such as cities or even countries as in the case of smart grids. This increase in the number of available sensors, actuators and data has led to an increase in the controllers' required computational capacity for such applications. 
One solution that has gained important traction is to outsource the computations to powerful remote servers, generically called clouds. Cloud-based computation services offer the potential to aggregate and to process information that is gathered from a large number of distributed agents, in domains such as machine learning, signal processing, smart grids, autonomous vehicles and control over the Internet-of-Things, with less overhead than distributed computations. Notwithstanding, there are substantial challenges that arise from the privacy-sensitive nature of some of the collected data. Secure communication to and from the cloud is a crucial aspect, but is not sufficient to protect the users' data. Recent examples such as data leakage and abuse by aggregating servers~\cite{Subashini11,Fernandes14,DataBreaches} have drawn attention to the risks of storing data in the clear and urged measures against an untrustworthy cloud. A plethora of different schemes addressing this challenge have been proposed in the literature, spanning many research areas. 

Schemes that allow secure function evaluation on private inputs from multiple parties while preserving the privacy of the data involved have been a desirable technique for many decades, coining the term \textit{secure multi-party computation}. Some secure multi-party computation protocols provide \textit{computational security}, while other provide \textit{information theoretic security}. 
Computational security guarantees that there is no computationally efficient algorithm the parties involved in the protocol can implement so that they gain private information beyond what can be inferred from their inputs and outputs of the protocol only. On the other hand, information theoretic security guarantees that no algorithms of unlimited computational power exist that can lead to private data leakage. One approach within this context is to encrypt the computation itself. This primarily refers to Yao's \textit{garbled circuits}, introduced in~\cite{Yao82}, and the protocol of Goldreich-Micali-Widgerson, introduced in~\cite{GMW87}, where both allow secure evaluation of functions represented as Boolean circuits. Although powerful with respect to the functions they can evaluate, garbled circuits require constructing high depth circuits and transmitting them over a communication channel, which incurs a substantial cost in computation and communication. While many improvements in the efficiency of garbled circuits have been made over the years, the complexity still prohibits the usage of such schemes in real time applications~\cite{Bellare12}. 

Another path within secure multi-party computation with computational security guarantees is computing on encrypted data. 
This approach, suggested in~\cite{Rivest78data} under the name ``privacy homomorphisms'', is currently known as homomorphic encryption. Homomorphic encryption allows evaluating specific types of functionals on encrypted inputs. 
Homomorphic encryption comes in three flavors: \textit{fully}, \textit{somewhat} and \textit{partially} homomorphic encryption. Fully homomorphic encryption allows evaluation of Boolean functions over encrypted data and was introduced in~\cite{GentryPhD}, and further improved in terms of efficiency in~\cite{Gentry13,DGHV10,Brakerski14,Cheon2017homomorphic}. However, the computational overhead is prohibitive, due to the complexity of the cryptosystem primitives and size of the encrypted messages~\cite{Halevi17}. 
Somewhat homomorphic encryption schemes, e.g.~\cite{Brakerski11}, are less complex, but support only a limited number of operations (additions and multiplications). Partially homomorphic encryption schemes are tractable but can support either multiplications between encrypted data, such as El Gamal~\cite{ElGamal84}, unpadded RSA~\cite{Rivest78data} or additions between encrypted data, such as Paillier~\cite{Paillier99}, Goldwasser-Micali~\cite{GM82}, DGK~\cite{DGK07}. 

While the above techniques ensure that the output is correctly computed without revealing anything about the inputs in the process, depending on the application, an adversary might be able to learn information about the input only from the output. Techniques such as \textit{differential privacy} ensure that the revealed output does not break the privacy of the input. 
Differential privacy, introduced in~\cite{Dwork2006calibrating} and surveyed in~\cite{Dwork08,Dwork2014algorithmic}, adds controlled randomness in the data, so that the results of statistical queries do not tamper with the privacy of the individual entries in a database, by introducing uncertainty in the output. However, the noise determines a trade-off between the accuracy of the results of the differentially private queries and the privacy parameter.

\subsection{Related work}
Homomorphic encryption has been recently used to design encrypted controllers, with partially homomorphic encryption schemes for linear controllers in~\cite{Kogiso15,Farokhi17}, model predictive control applications in~\cite{Darup18,Alexandru18}, with fully homomorphic encryption in~\cite{Kim16}, and to design observers in~\cite{Alanwar17}. 
Unconstrained optimization problems are commonly used in machine learning applications, e.g., for linear regression, and several works have addressed gradient methods with partially homomorphic encryption~\cite{Han10grad,Hardy17} or with ADMM in a distributed setting~\cite{Zhang17privacy}. However, adding constraints substantially complicates the optimization problem by restricting the feasible domain. In addition, constraints introduce nonlinear operations in the optimization algorithm, which cannot be handled inherently by the partially homomorphic cryptosystems. In our prior work~\cite{Shoukry16,Alexandru17}, we showed how to overcome these difficulties by using blinded communication in a centralized setup. Recently,~\cite{Lu18} proposed a distributed approach to the constrained optimization problem with homomorphic encryption, where each party locally performs the projection on the feasible domain on unencrypted data. Similarly,~\cite{Zheng2019helen} proposed a distributed secure multi-party solution to regularized linear model training that also allows for malicious adversaries. 

Examples of works that tackle private linear programming problems or machine learning applications with optimization subroutines using garbled circuits, secret sharing, and hybrid approaches between the above are~\cite{Li2006secure,Xie16,Mohassel17,Gascon17,Chase17} and the references within. 
In control applications, works such as~\cite{leNy14differentially,Dong2015differential,Cortes2016differential,Wang2017differential} employ differential privacy to protect dynamic data. 
Furthermore, examples such as~\cite{Huang15,Han17,Nozari18} make use of differential privacy techniques in optimization algorithms. Other lines of work solve optimization problems with \mbox{non-cryptographic} methods: using transformations, e.g.~\cite{Vaidya2009privacy,Weeraddana2013per,Wang2015efficient,Salinas2016efficient}, which guarantee an uncertainty set around the private data, or using arbitrary functions for obfuscating the objective functions in a special graph topology~\cite{Gade16}. 

\subsection{Contributions}

In this paper we consider a privacy-preserving cloud-based constrained quadratic optimization. In control theoretic applications and machine learning, linear and quadratic optimization problems arise frequently -- e.g., state estimation under minimum square error, model predictive control, support vector machines -- which require privacy guarantees of the data involved in the computation. 

We develop a new tractable optimization protocol to privately solve constrained quadratic problems. Our protocol relies on cryptographic tools called encryption schemes. To solve the optimization problem on encrypted data, we use an additively homomorphic encryption scheme, where, in short, \textit{addition commutes with the encryption function}. Thus, a party can process encrypted data without having direct access to the data. 
The novelty is how to handle in a privacy-preserving manner the constraints of the problem that introduce non-linearities that cannot be supported by additively homomorphic encryption schemes. We show that a projected gradient method that operates on the Lagrange dual problem can alleviate this problem and can be efficiently run on encrypted data by exploiting communication between the participating parties. 

The main contributions of our work are the following:
\begin{itemize}[topsep=0pt,wide, labelwidth=!, labelindent=0pt]
	\item We formally state and prove computational security guarantees for such protocol. 
The proof relies on applying cryptographic tools to the specific optimization algorithm that runs over multiple iterations. 
	\item We propose an alternative protocol which sacrifices some privacy but involves less communication overhead. 
	\item We implement the protocols and show the computational~and communication complexity produce reasonable running times.
\end{itemize}

Furthermore, we emphasize the difference between the computational security guarantee with the \textit{non-unique retrieval} guarantee that is important in such optimization problems. 
We finally point out that the present manuscript provides detailed security proofs and analyses not available in previous versions of the papers~\cite{Shoukry16,Alexandru17}.

\emph{Organization.}
The rest of the paper is organized as follows: we formulate the problem in Section~\ref{sec:problem_setup} and formally introduce the computational security guarantees in Section~\ref{sec:privacy_goals}. We describe the cryptographic tools we use in Section~\ref{sec:encryption_sch}, more specifically, the partially homomorphic scheme, along with its properties and the symmetric encryption used. In Section~\ref{sec:gradient_ascent}, we describe the optimization theory behind the proposed protocol and justify the choice of the gradient method. Furthermore, we present the main protocol and the subroutines that compose it and in Section~\ref{sec:security_proof}, we show that the protocol achieves privacy of the agent's data and of the target node's result. We discuss possible relaxations of the security guarantees and propose a more efficient protocol under these weaker conditions in Section~\ref{sec:relaxation_security}. In Section~\ref{sec:discussion}, we provide a privacy analysis of the problem concerning the input-output relation. We present the complexity analysis of the protocols and show that the experimental results point out reasonable running times in Section~\ref{sec:complexity_time}. Additional tools necessary for proving the security of the proposed protocol are introduced in Appendix~\ref{app:cryptographic_prel} and further details on the protocols used are given in Appendix~\ref{app:comparison}. Finally, the detailed privacy proofs are given in Appendix~\ref{app:proof}.

\emph{Notation.}
We denote matrices by capital letters. $\mathbb R^{m\times n}$ represents the set of $m\times n$ matrices with real elements and $\mathbb S^n_{++}$ represents the set of symmetric positive definite $n\times n$ matrices. Element-wise inequalities between vectors are denoted by $\preceq$. The notation $x\leftarrow X$ means that $x$ is uniformly drawn from a distribution $X$. $\text{Pr}$ denotes the probability taken over a specified distribution. We refer to algorithms that are run interactively by multiple parties as protocols.


\section{Problem Setup}\label{sec:problem_setup}

\subsection{Motivating Examples}
Quadratic optimization is a class of problems frequently employed in control system design and operation. As a first motivating example, consider the problem of estimating the state of a dynamical system from privacy-sensitive sensor measurements. For instance, such a problem arises in smart houses where the temperature or energy readings of the sensors are aggregated by a cloud controller and can reveal whether the owners are in the building. In particular, let the system dynamics and sensor measurements be described by:
\begin{align}
\begin{split}
	x_{t+1} &= A \, x_t + w_t  \\
	y_t &= C \, x_t + v_t, 
\end{split}
\end{align}
for $t =0,\ldots, T-1$, where $w_t, v_t$ are process and measurement noise respectively. The system and sensor parameters $A \in \mathbb{R}^{n \times n},C\in \mathbb{R}^{p \times n}$ can be thought as publicly available information while the sensor measurements $y_0, \dots, y_{T-1}$ are privacy-sensitive data. The untrusted cloud has to collect the measurements and output an initial state estimate $x_0$ of the system to a target agent, while maintaining the privacy of the sensor data and final output. A simple state estimate may be found as the solution to the least squares problem:
\begin{align}\label{eq:estimation_example}
	&\underset{x_0 \in \mathbb{R}^n }{\text{min}} \quad \frac{1}{2} \sum_{t=0}^\intercal \| y_t - C A^t\, x_0 \|^2_2  = 
	\frac{1}{2} \| {y} - \mathcal{O}\, x_0 \|^2_2,
\end{align}
where
\begin{equation} 
{y} = \left[ \begin{array}{c}
y_0 \\ y_1\\ \vdots \\ y_{T-1}
\end{array}\right], 
\; 
\mathcal{O} = \left[ \begin{array}{c}
C\\  CA\\ \vdots\\ CA^{T-1}
\end{array}\right]
\end{equation}
and $\mathcal{O} \in \mathbb{R}^{Tp \times n}$ is the well-known system observability matrix. 
More general state estimation problems may also include constraints, e.g., the initial state may lie in a polyhedral set $D x_0 \preceq b$ where the shape of the polyhedron captured by the matrix $D$ is publicly known but its position and size captured by the vector $b$ is private information.

As a second example, consider a control problem with privacy-sensitive objectives. Suppose we are interested in steering a dynamical system: 
\begin{align}\label{eq:system}
\begin{split}
	x_{t+1} &= A \, x_t +B \, u_t \\
	y_t &= C \, x_t,
\end{split}
\end{align}
starting from a private initial position $x_0$ while guaranteeing safety state constraints of the form:
\begin{align}
	\underline{x}_t \preceq x_t \preceq \overline{x}_t, t =1,\ldots, T,
\end{align}
where $\underline{x}_t , \overline{x}_t$ for $ t =1,\ldots, T $ are private values. 
Such problems arise when exploring vehicles are deployed in uncertain or hazardous environments or when different users compete against each other and want to hide their tactics. 

Denote by $x^r$ and $u^r$ the private values such that the system is guaranteed to track a reference~$r$. Then, we need to solve the problem:
\begin{align}\label{eq:control_example_1}
\begin{split}
	\underset{x_k \in \mathbb{R}^n, u_k \in \mathbb{R}^m }{\text{min}}~&~\frac{1}{2} \sum_{t=0}^{T-1} \|y_t-C\,x^r_t\|^2_Q + \|u_t-u^r_t\|^2_R \\
	\text{s.t.}~&~ x_{t+1} = A \, x_t +B \, u_t, \quad y_t = C\, x_t \\
	 &~\underline{x}_t \preceq x_t \preceq \overline{x}_t, \quad t =1,\ldots, T,
\end{split}
\end{align}
where $Q,R$ are positive definite cost matrices and we used the quadratic norm notation $\|z\|^2_P := z^\intercal P z$.

The publicly known system dynamics \eqref{eq:system} are equality constraints which can be eliminated to obtain a control input design problem with only linear inequality constraints:
\begin{align}\label{eq:control_example_2}
\begin{split}
	\underset{x_k \in \mathbb{R}^n, u_k \in \mathbb{R}^m }{\text{min}}~&~ \frac{1}{2}\sum_{t=0}^{T-1} u_t^\intercal \, H_t\,u_t + [\begin{array}{ccc} x_0 & x^r_t & u^r_t \end{array}] \,F_t \, u_t \\
	\text{s.t.}~&~\underline{x}_t \preceq A^t\, x_0 +\sum_{j=0}^{t-1} A^{t-1-j} B\, u_j  \preceq \overline{x}_t, \\
	&~t =1,\ldots, T,
\end{split}
\end{align}
with $H_t $ and $F_t$ are matrices that depend on the costs $Q,R$ and the system's dynamics~\eqref{eq:system} appropriately computed from~\eqref{eq:control_example_1}.

\subsection{Problem statement} 

The above examples can be modeled as constrained quadratic optimization problems with distributed private data. We consider three types of parties involved in the problem: a number of agents~$\mc A_i$, $i=1,\ldots,p$, a cloud server~$\mc C$ and a target node~$\mc T$. The purpose of this setup is to solve an optimization problem with the data provided by the agents and the computation performed on the cloud and send the result to the target node. The architecture is presented in Figure~\ref{fig:setup}.

Let us consider a strictly-convex quadratic optimization problem, which we assume to be feasible:
\begin{align}\label{eq:problem}
	\begin{split}
	x^\ast ~= ~\argmin\limits_{x\in\mathbb R^n}&~\frac{1}{2} x^\intercal \, Q \cloud \, x + c\agent^\intercal x\\
	 s.t.&~ A\cloud \, x\preceq b\agent,
	 \end{split}
\end{align}
where the variables and the parties to which they belong to are described as follows:\\
\noindent\textbf{Agents $\mc A = (\mc A_1,\ldots,\mc A_p)$}: The agents are parties with low computational capabilities that possess the private information $b\agent$ and $c\agent$. The private information is decomposed across the agents as: $b\agent = (b_1,\ldots,b_p)$ and $c\agent = (c_1,\ldots,c_p)$, with $b_i\in \mathbb R^{m_i}$ and $c_i\in \mathbb R^{n_i}$ being the private data of agent $i$ such that $\sum_{i=1}^p m_i = m$ and $\sum_{i=1}^p n_i = n$.\\
\noindent\textbf{Cloud $\mc C$}: The cloud is a party with high computational capabilities that has access to the matrices $Q\cloud\in \mathbb S^n_{++}$ and $A\cloud\in\mathbb R^{m\times n}$. 
When the computation is sophisticated and/or involves proprietary algorithms, $Q\cloud$ and $A\cloud$ are private data of the cloud. In order to capture a greater number of problems, we will also consider the case where $Q\cloud$ or $A\cloud$ are public.
\\
\noindent\textbf{Target Node $\mc T$}: The target node is a party with more computational capabilities than the agents that is interested in the optimal solution $x^\ast$ of the problem.  
The target node can be either one of the agents or a separate cloud server.

\begin{figure}[h]
  \centering
    \includegraphics[width=0.38\textwidth]{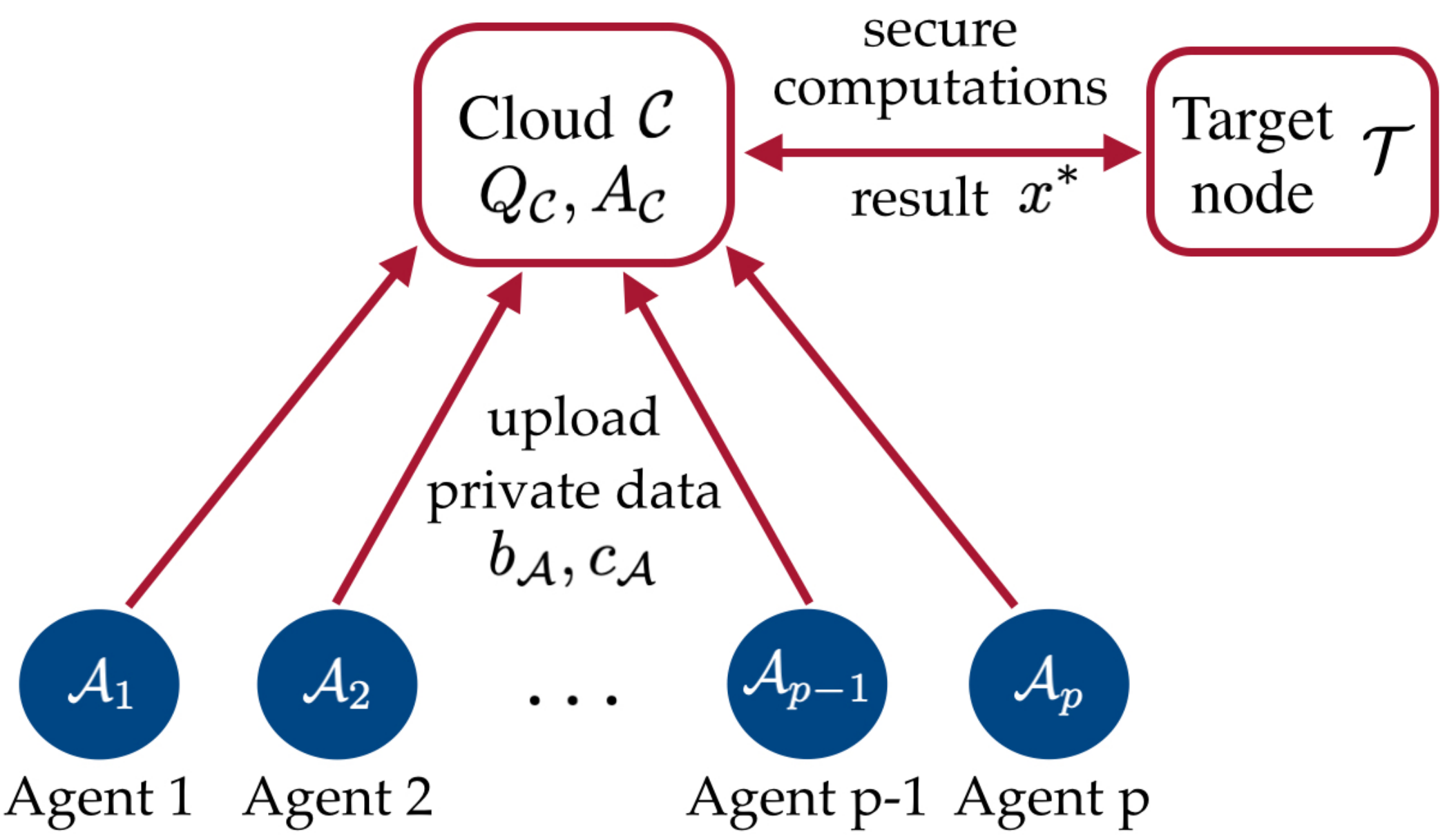}
  \caption{
  	Architecture of the problem: Agents are low-resource parties that have private data that they outsource to a powerful server, called the cloud. The cloud has to solve an optimization problem on the private data of the agents and send the result to a party called the target node, which will help with the computations.} 
   \label{fig:setup}
\end{figure}

Note that in the first motivating example~\eqref{eq:estimation_example}, the matrix $Q\cloud$ in~\eqref{eq:problem} corresponds to the publicly known matrix $\mathcal{O}^\intercal \mathcal{O}$ and the vector $c_{\mathcal{A}}$ corresponds to the private vector $-\mathcal{O}^\intercal \, {y}$. For the objective to be strongly convex, we require that $\text{rank}(\mathcal{O})=n$, which also implies standard system observability. Similarly, in the second example~\eqref{eq:control_example_2}, the matrix $Q\cloud$ is composed from the regulating cost matrices, the cost vector $c_{\mathcal{A}}$ is formed from the private initial conditions and steady-state solution for the reference tracking, mixed by the system's dynamics, and the constraints vector $b_{\mathcal{A}}$ depends on the private state bounds and initial condition.

\subsection{Adversarial model}

In most cloud applications, the service provider has to deliver the contracted service or otherwise the clients would switch to another service provider. This means that the cloud cannot alter the data it receives. Moreover, the clients' interest is to obtain the correct result from the service they pay for, hence, the agents and target node will also not alter the data. However, the parties can locally process the data they receive in any fashion they want. This model is known as \mbox{semi-honest}, which is defined as follows:

\begin{definition}(\mbox{Semi-honest model})
A party is \mbox{semi-honest} if it does not deviate from the steps of the protocol but may store the transcript of the messages exchanged and process the data received in order to learn more information than stipulated by the protocol.
\end{definition}

This model also holds when considering eavesdroppers on the communication channels. An adversarial model that only considers eavesdroppers as adversaries differs from the \mbox{semi-honest} model by the fact that, apart from the untrusted channels, the parties that perform the computations are also not trusted. Malicious and active adversaries -- that diverge from the protocols or tamper with the messages -- are not considered in this paper.

The purpose of the paper is to solve Problem~\eqref{eq:problem} using a secure multi-party computation protocol for semi-honest parties. This protocol takes as inputs the private data of the agents, as well as the cloud's data, and involves the parties in exchanging messages and participating in some specified computations, and eventually outputs to the target node the solution of the optimization problem. This protocol should guarantee \textit{computational security}. More specifically, the cloud cannot obtain any information about the private inputs of the agents and the output of the target node and similarly, the target node cannot obtain any information about the private inputs of the agents and the cloud, other than what they can compute using their inputs and outputs and public information, by running a computationally efficient algorithm after the execution of the protocol.
%


\section{Privacy goals}\label{sec:privacy_goals}

We introduce some notion necessary to the security definitions. In what follows, $\{0,1\}^\ast$ defines a sequence of bits of unspecified length. An ensemble $X = \{X_n\}_{n\in\mathbb N}$ is a sequence of random variables ranging over strings of bits of length polynomial in $n$, denoted by $\text{poly}(n)$, arising from distributions defined over a finite set $\Omega$. By positive polynomials, we refer to the set of polynomials that on positive inputs have positive values.

\begin{definition}(Statistical indistinguishability~\cite[Ch.~3]{Goldreich03foundations|})\label{def:stat_ind}
The ensembles $X=\{X_n\}_{n\in\mathbb N}$ and $Y=\{Y_n\}_{n\in\mathbb N}$ are \textbf{statistically indistinguishable}, denoted $\stackrel{s}{\equiv}$, if for every positive polynomial $p$, and all sufficiently large $n$:
\[\frac{1}{2}\sum_{\alpha\in\Omega} \left |\text{Pr}[X_n = \alpha] - \text{Pr}[Y_n = \alpha] \right |< \frac{1}{p(n)},\]
where the quantity on the left is called the statistical distance between the two ensembles.
\end{definition}

It can be proved that two ensembles are statistically indistinguishable if no algorithm can distinguish between them. Computational indistinguishability is a weaker notion of the statistical version, as follows:

\begin{definition}\label{def:comp_ind}(Computational Indistinguishability~\cite[Ch.~3]{Goldreich03foundations|})
The ensembles $X=\{X_n\}_{n\in\mathbb N}$ and $Y=\{Y_n\}_{n\in\mathbb N}$ are \textbf{computationally indistinguishable}, denoted $\stackrel{c}{\equiv}$, if for \textbf{every probabilistic polynomial-time} algorithm $D:\{0,1\}^\ast\rightarrow \{0,1\}$, called the distinguisher, every positive polynomial $p$, and all sufficiently large $n$, the following holds:
\[\big |\text{Pr}_{x\leftarrow X_n} [D(x) = 1] - \text{Pr}_{y\leftarrow Y_n}[D(y) = 1] \big| < \frac{1}{p(n)}.\]
\end{definition}

The previous definition is relevant when defining the privacy goals to be guaranteed by the protocol: \textbf{\mbox{two-party} privacy} of the sensitive data of the parties. 
The intuition is that a protocol privately computes a functionality if nothing is learned after its execution, i.e., if all information obtained by a party after the execution of the protocol (while also keeping a record of the intermediate computations) can be obtained only from the inputs and outputs available to that party. 

\begin{definition}\label{def:view}(Two-party privacy w.r.t. \mbox{semi-honest} behavior~\cite[Ch.~7]{Goldreich04foundations||}) 
 Let $f:\{0,1\}^\ast \times \{0,1\}^\ast \rightarrow \{0,1\}^\ast \times \{0,1\}^\ast$ be a functionality, and $f_1(x_1,x_2)$, $f_2(x_1,x_2)$ denote the first and second components of $f(x_1,x_2)$, for any inputs $x_1,x_2\in\{0,1\}^\ast$. Let $\Pi$ be a \mbox{two-party} protocol for computing $f$. The \textbf{view} of the $i$-th party ($i=1,2$) during an execution of $\Pi$ on the inputs $(x_1,x_2)$, denoted $V^\Pi_i (x_1,x_2)$, is $(x_i,coins,m_1,\ldots,m_t)$, where $coins$ represents the outcome of the $i$'th party's internal coin tosses, and $m_j$ represents the $j$-th message it has received. For a deterministic functionality $f$, we say that $\Pi$ \textbf{privately computes} $f$ if there exist probabilistic polynomial-time algorithms, called \textbf{simulators}, denoted $S_i$, such that:
 \[\{S_i(x_i,f_i(x_1,x_2)) \}_{x_{1,2}\in\{0,1\}^{\ast}} \stackrel{c}{\equiv} \{V_i^\Pi (x_1,x_2)\}_{x_{1,2}\in\{0,1\}^\ast}.\]
 \end{definition}
  
This definition assumes the correctness of the protocol, i.e., the probability that the output of the parties is not equal to the result of the functionality applied to the inputs is negligible~\cite{Lindell17}. 

For protocols that involve more than two parties, Definition~\ref{def:view} can be extended to \textbf{multi-party privacy}, by also taking into consideration coalitions of semi-honest parties. The next definition states that multi-party protocol privately computes the functionality it runs if everything that a coalition of parties can obtain from participating in the protocol and keeping records of the intermediate computations can be obtained only from the inputs and outputs of these parties.

\begin{definition}\label{def:coalition_view}(Multi-party privacy w.r.t. \mbox{semi-honest} behavior~\cite[Ch.~7]{Goldreich04foundations||}) 
 Let $f:(\{0,1\}^\ast)^n \rightarrow (\{0,1\}^\ast)^n$ be a $n$-ary functionality, where $f_i(x_1,\ldots,x_n)$ denotes the $i$-th element of $f(x_1,\ldots,x_n)$. Denote the inputs by $\bar x = (x_1,\ldots,x_n)$. For $I=\{i_1,\ldots,i_t\}\subset [n]=\{1,\ldots,n\}$, we let $f_I(\bar x)$ denote the subsequence $f_{i_1}(\bar x),\ldots,$ $f_{i_t}(\bar x)$, which models a coalition of a number of parties. Let $\Pi$ be a $n$-party protocol that computes $f$. The \textbf{view} of the $i$-th party during an execution of $\Pi$ on the inputs $\bar x$, denoted $V^\Pi_i (\bar x)$, is defined as in Definition~\ref{def:view}, and we let the view of a coalition be denoted by $V_I^\Pi (\bar x) = (I,V_{i_1}^\Pi(\bar x),\ldots,V_{i_t}^\Pi(\bar x))$. For a deterministic functionality $f$, we say that $\Pi$ \textbf{privately computes} $f$ if there exist simulators $S$, such that, for every $I\subset [n]$, it holds that, for $\bar x_t = (x_{i_1},\ldots,x_{i_t})$:
 \begin{align*}
\{S(I,(\bar x_t),f_I(\bar x_t)) \}_{\bar x \in (\{0,1\}^\ast)^n} &\stackrel{c}{\equiv} 
\{V_I^\Pi (\bar x)\}_{\bar x \in (\{0,1\}^\ast)^n}.
\end{align*}
\end{definition}

Further underlying details related to the cryptographic notion of security and its interpretation are presented in Appendix~\ref{subsec:appendix_security}.
%


\section{Encryption scheme}\label{sec:encryption_sch}

\subsection{Additively Homomorphic Encryption}\label{subsec:ahe}

Let $E(\cdot)$ define a generic encryption primitive, with domain the space of private data, called plaintexts, and codomain the space of encrypted data, called ciphertexts. For probabilistic encryption schemes, the encryption primitive also takes as input a random number. The decryption primitive $D(\cdot)$ is defined on the space of ciphertexts and takes values in the space of plaintexts. Additively homomorphic schemes satisfy the property that there exists an operator $\oplus$ defined on the space of ciphertexts such that: 
\begin{equation}~\label{eq:abstract_PHE}
E(a)\oplus E(b) \subset E(a+b),
\end{equation}
for any plaintexts $a,b$ supported by the scheme. Here, we use set inclusion instead of equality because the encryption of a message is not unique in probabilistic cryptosystems. 
Intuitively, equation~\eqref{eq:abstract_PHE} means that performing this operation on the two encrypted messages, we obtain a ciphertext that is equivalent to the encryption of the sum of the two plaintexts. Formally, the decryption primitive $D(\cdot)$ is a homomorphism between the group of ciphertexts with the operator $\oplus$ and the group of plaintexts with addition $+$, which justifies the name of the scheme. It is immediate to see that if a scheme supports addition between encrypted messages, it will also support subtraction, by adding the additive inverse, and multiplication between an integer plaintext and an encrypted message, obtained by adding the encrypted messages for the corresponding number of times. 

In this paper, we use the popular Paillier encryption, but any other additively homomorphic encryption scheme can be employed. The \textbf{Paillier cryptosystem}~\cite{Paillier99} is an asymmetric additively homomorphic encryption scheme. Asymmetric or public key cryptosystems involve a pair of keys: a public key that is disseminated publicly, and which is used for the encryption of the private messages, and a private key which is known only to its owner, used for the decryption of the encrypted messages. In this section, we will only present the notation for the encryption and for the allowed operations between the encrypted messages and postpone the implementation details to Appendix~\ref{subsec:appendix_paillier}. 

In the Paillier encryption scheme, the private data are elements of the ring of integers modulo $N$, denoted by $\mathbb Z_N$, where $N$ is a large integer of $\sigma$ bits, called the Paillier modulus. The encrypted messages take values in the positive integers modulo $N^2$, denoted by $\mathbb Z^\ast_{N^2}$. Paillier is a probabilistic encryption scheme, which means that the encryption primitive also takes as an argument a random number which, for simplicity, we avoid in this notation. For a plaintext message $a$, we denote the Paillier encryption by $[[a]]:=E(a)$, for an instance of the random number. For readability, we will use throughout the paper the following abstract notation for the operations on the encrypted space:
\begin{align}\label{eqref:hom_prop}
\begin{split}
	[[a]]\oplus[[b]] &\stackrel{d}{=} [[a+b]]\\
	b\otimes[[a]] &\stackrel{d}{=} [[ba]],~~ \text{ for plaintexts } a,b\in\mathbb Z_N,
\end{split}
\end{align}
where $\stackrel{d}{=}$ means that the equality holds after applying the decryption primitive on both sides~\cite{Paillier99}. The expressions for the abstracted operations are described in Appendix~\ref{subsec:appendix_paillier}.
We will slightly abuse these notations to denote encryptions of vectors and additions and multiplication by vectors and matrices. In our protocol, the target node $\mc T$ is the owner of a pair of Paillier keys. Hence, for everything that follows, $[[\cdot]]$ denotes the encryption with the target node's public key $pk\target$. 

\emph{Example: Private unconstrained quadratic optimization.} 
In order to illustrate the usage of the Paillier encryption scheme for computations on encrypted data, consider the following example of an unconstrained optimization problem:
\begin{equation}~\label{eq:unconstrained}
	\min\limits_{x} f(x) = \frac{1}{2}x^\intercal Q\cloud x + c\agent^\intercal x,
\end{equation}
where the variables have the same meaning as in Problem~\eqref{eq:problem} described in Section~\ref{sec:problem_setup}. The cloud has access to the matrix $Q\cloud$. If the cloud also had access to private data $c\agent$ from the agents, the gradient method could be employed to obtain the optimal unencrypted solution of Problem~\eqref{eq:unconstrained}. An iteration of the gradient method has the form:
\[x_{k+1} = x_k - \eta\nabla f(x_k) = x_k - \eta Q\cloud x_k - \eta c\agent,\]
where $\eta>0$ is the step-size chosen by the cloud and $k=0,\ldots,K-1$ for a number of iterations $K$. However, in order to protect their private data and the privacy of the solution $x^\ast$, the agents send the value $c\agent$ encrypted to the cloud as $[[c\agent]]$. 
It is very important to notice that \textit{for a quadratic objective function, only linear operations in the private data $x_k$ and $c\agent$ are required in order to compute one iteration of the gradient descent algorithm}. Hence, by taking advantage of the additively homomorphic property~\eqref{eqref:hom_prop}, the cloud can locally compute the optimal solution by the gradient descent in the encrypted domain, for all $k = 0,\ldots,K-1$:
\begin{equation}\label{eq:unconstrained_iter}
[[x_{k+1}]] = (I-\eta Q\cloud)\otimes[[x_k]]\oplus(-\eta)\otimes [[c\agent]],
\end{equation}
and send the final point $[[x_K]]$ to the target node to decrypt and obtain the desired result. 
Such a protocol satisfies the desired security according to Definition~\ref{def:view} provided in Section~\ref{sec:privacy_goals} by the fact that the cloud only has access to data encrypted by a semantically secure encryption scheme, such as Paillier~\cite{Paillier99}. More details are provided in Section~\ref{sec:security_proof}.

However, the optimization problem considered in this paper is a quadratic problem with constraints, which introduce nonlinear operations in the optimization algorithm and cannot be handled inherently by the Paillier cryptosystem. We will describe in Section~\ref{sec:gradient_ascent} how we leverage communication between the parties in order to achieve privacy-preserving nonlinear operations.

\subsection{Symmetric encryption scheme}\label{subsec:symmetric}
Apart from the (public key) partially homomorphic cryptosystem above, we will also use a specific symmetric key cryptosystem for reasons that will become apparent in the protocol, related to hiding the data from the target node. This symmetric key cryptosystem is similar to the well-known one-time pad~\cite{Vernam1926,Bellovin11}, and can be thought of as additively blinding a private value by noise. In contrast to public key cryptography, symmetric key algorithms perform the encryption and decryption with the same key. For messages of $l$ bits, the key is generated as $sk\in \mathbb Z$ with length of $\lambda+l$ bits, where $\lambda$ is the security parameter. The encryption primitive is $E'(a) = a + sk$, with $a\in [0,2^l)\cap \mathbb Z$, and the decryption is obtained as $a = D'(E'(a)) = E'(a) - sk$. This is also called two-out-of-two additive secret sharing~\cite{Dodis2007lecture} when the addition is on an abelian group. The security of this scheme lies on generating a uniformly random key and on using this key for encryption only once, which yields that the distribution of $E'(a)$ is statistically indistinguishable from a random number sampled of $l+\lambda+1$ bits. For this reason, $\lambda$ is called the statistical security parameter. We also notice that \textit{this symmetric cryptosystem is compatible with the Paillier cryptosystem, in the sense that the two encryptions commute}: $D'(E(E'(a)))) = E(a)$ by using $E(sk)$ instead of $sk$ for decryption, where $E'(a)$ is performed on the message space of the Paillier cryptosystem. Such a scheme is commonly employed for blinding messages, for instance in~\cite{Veugen10,Jeckmans13,Bost15}.

For simplicity, in the protocols described in Section~\ref{sec:gradient_ascent}, we will directly use the summation with a random number to depict the above symmetric key cryptosystem, respectively the difference by the same random number for decryption.

The strength of a cryptosystem relies on the computational intractability of retrieving the private key from the public information -- an adversary holding the public information cannot find the private key by brute force computations. Public key cryptosystems like the Paillier cryptosystem involve a security parameter $\sigma$, which is usually the length of the modulus of the public key, and symmetric key cryptosystems involve a security parameter $\lambda$, which is the bit-length of the key. As is common in practice, increasing the size of the security parameters to values larger than, say 1024 bits for $\sigma$ and 80 bits for $\lambda$, increases the security of the system against any computationally efficient inference attack.


\section{Secure Constrained Quadratic Optimization}\label{sec:gradient_ascent}

In this section we introduce our main algorithmic results. Let us first describe the optimization algorithm used for solving the minimization problem~\eqref{eq:problem} on unencrypted data.

For strongly convex problems, one can resort to duality theory~\cite[Ch.~5]{BV04convex} to compute the projection on the feasible set, and be able to retrieve the optimal value of the primal problem from the optimal value of the dual problem. 
For the quadratic optimization problem \eqref{eq:problem}, its dual is also a quadratic optimization problem:
\begin{align*}
	\mu^\ast = \argmax\limits_{\mu\in\mathbb R^m}&~-\frac{1}{2} (A\cloud^\intercal \mu + c\agent)^\intercal Q\cloud^{-1}(A\cloud^\intercal\mu + c\agent) - \mu^\intercal b\agent\\
	 s.t.&~ \mu \succeq 0.\numberthis\label{eq:dual}
\end{align*}
The dual objective function is denoted by $g(\mu)$ and its gradient is equal to: 
\begin{equation}\label{eq:grad}
\nabla g(\mu) = -A\cloud Q\cloud^{-1} (A\cloud^\intercal\mu + c\agent) - b\agent.
\end{equation}

Under standard constraint qualifications, e.g., Slater's condition~\cite[Ch.~5]{BV04convex}, strong duality between the primal and dual holds, which means the optimal objective in the primal problem~\eqref{eq:problem} is equal to the objective in the dual problem~\eqref{eq:dual}. Moreover, the optimality conditions (Karush-Kuhn-Tucker) hold and are the following:
\begin{align}
	Q\cloud x^\ast + A\cloud^\intercal \mu^\ast + c\agent &= 0 \label{eq:KKT1}\\
	A\cloud x^\ast - b\agent &\preceq 0,~~\mu^\ast \succeq 0 \label{eq:KKT23}\\
	\mu_i^\ast(a_i^\intercal x^\ast - b_i) &= 0, ~i = 1,\ldots,m \label{eq:KKT4}.
\end{align}

For strictly convex problems, i.e., $Q\cloud\in\mathbb S^n_{++}$, the optimal solution of the primal problem can be obtained from \eqref{eq:KKT1} as
\[x^\ast = - Q\cloud^{-1}(A\cloud^\intercal \mu^\ast + c\agent).\]

An algorithm for computing the optimum in problem~\eqref{eq:dual}, which we will show is also compatible with the partially homomorphic encryption of Section~\ref{sec:encryption_sch}, is the projected gradient ascent method. The projected gradient ascent is composed by iterations of the following type:
\begin{equation}\label{eq:projection}
	\mu_{k+1}=\max\{\mb 0,\mu_k + \eta\nabla g(\mu_k)\},
\end{equation}
where $\eta>0$ is the step size and $\mu_{k+1}$ is the projected value of $\mu_k + \eta\nabla g(\mu_k)$ over the non-negative orthant. For full rank of $A\cloud Q^{-1}\cloud A\cloud^\intercal$, the dual problem is strictly convex and the algorithm converges with a linear rate~\cite{nesterov13book} for a fixed step size $\eta = \frac{1}{L}$, where $L = \lambda_{max}(A\cloud Q^{-1}\cloud A\cloud^\intercal)$. For \mbox{non-strictly} convex dual function, the gradient ascent algorithm converges in sublinear time~\cite{nesterov13book}. 

\subsection{Projected gradient ascent on encrypted data}\label{sec:encrypted_gradient}

As stated in Section~\ref{sec:problem_setup}, we aim to solve an optimization problem outsourced to the cloud on private distributed data from the agents and send the result to the target node. To protect the agents' data, we use an encryption scheme that allows the cloud to perform linear manipulations on encrypted data, as described in Section~\ref{sec:encryption_sch}. To this end, the target node generates a pair of keys $(pk\target, sk\target)$ and distributes the public key to the agents and the cloud, enabling them to encrypt their data, which only the target node will be able to decrypt, using the private key. We consider that all the data is represented on integers of~$l$ bits and comment on this further in Section~\ref{subsec:quantization}.

The main difficulty in performing the projected gradient ascent on encrypted data is performing iteration~\eqref{eq:projection}. We have already seen in the example in Section~\ref{subsec:ahe} that the update of the iterate in the direction of the gradient ascent can be computed locally by the cloud directly on the encrypted data~\eqref{eq:unconstrained_iter}. However a first challenge lies in performing the comparison with zero. Due to the probabilistic nature of the Paillier encryption scheme, \textit{the order on the plaintext space is not preserved on the ciphertext space} and comparison on encrypted data cannot be performed locally by the cloud. Moreover, after the comparison is performed, the update of the encrypted iterate~\eqref{eq:projection} has to be done in a private way, so that the result of the maximum operation is not revealed to any of the parties involved (the cloud and the target node). These two steps are the main computational bottleneck in the protocol we propose, as both require secure communication between the cloud and the target node. A preliminary version of this solution was presented in~\cite{Alexandru17}.

We can privately achieve the two steps mentioned above in three stages. First, the cloud has to randomize the order of the two encrypted variables it wants to compare (Protocol~\ref{alg:rand_step}). Second, the cloud and target engage in an interactive comparison protocol that takes as inputs the two randomized variables and outputs the result of the comparison to the target node (Protocol~\ref{alg:comp_DGKV}). Third, the update of the dual iterate is achieved through an interactive protocol between the cloud and target node, which takes as inputs the two randomized variables and the result of the comparison and outputs to the cloud the updated iterate (Protocol~\ref{alg:update_step}). Throughout this paper, by comparison we mean element-wise comparison, since the variable $\mu$ is a vector.

\subsection{Secure comparison protocol}\label{subsec:comparison}
In order to privately compute~\eqref{eq:projection}, i.e., hide the result from \textit{all} the parties involved, we want to keep the result of the comparison of the updated iterate $\mu_k+\eta\nabla g(\mu_k)$ with zero unknown to both the cloud and the target node. The comparison protocol will reveal the result of the comparison between the two inputs to the target node $\mc T$. However, if we introduce an additional step where $\mc C$ randomizes the order of the two values that it wants to compare, then $\mc T$ does not learn any information by knowing the result of the comparison.

\begin{protocol}[Randomization step]
  \label{alg:rand_step}
  \begin{algorithmic}[1]
\small
 \Require{$\mc C$: $[[\bar \mu]],[[0]]$, where $\bar\mu:= \mu + \eta\nabla g(\mu)$}
 \Ensure{$\mc C$: $[[a]],[[b]]$}
 \State $\mc C$: choose a random permutation $\pi$ on two elements
\State $\mc C$: output $[[a]],[[b]]\leftarrow \pi([[0]],[[\bar\mu]])$
\end{algorithmic}
\end{protocol}

Next, we will demonstrate how the comparison protocol works via the following instantiation. Damg{\aa}rd, Geisler and Kr{\o}igaard introduced a protocol in~\cite{DGK07,DGK09correction} for secure comparison of two private inputs of different parties, which we will call the DGK protocol. To this end, they also propose an additively homomorphic encryption scheme, which has the property that checking if the value zero is encrypted is more efficient than simply decrypting, which is useful for comparisons and working with bits. An extension of this protocol to the case where none of the parties knows the two values to be compared, which is of interest to us, and some improvements in terms of efficiency were proposed in~\cite{Veugen12}. 

\begin{protocol}[Protocol for secure \mbox{two-party} comparison with two encrypted inputs using DGK \cite{DGK07,Veugen12}]
\label{alg:comp_DGKV}
\begin{algorithmic}[1]
\small
 \Require{$\mc C$: $[[a]],[[b]]$; $\mc T$: $sk\target,sk_{DGK}$}
 \Ensure{$\mc T$: bit $t$: $(t=1)\Leftrightarrow (a \leq b)$}
\State $\mc C$: choose random number $\rho$
\State $\mc C$: $[[z]] \gets [[b]]\oplus(-1)\otimes[[a]]\oplus[[2^l+\rho]]$, send $[[z]]$ to $\mc T$
\Statex \Comment{$z\gets b-a+2^l +\rho$}
\State $\mc T$: decrypt $[[z]]$
\State $\mc C$: $\alpha\gets \rho\mod 2^l$
\State $\mc T$: $\beta\gets z\mod 2^l$
\State $\mc C$,$\mc T$: perform a comparison protocol, e.g, DGK, such that $\mc C$ gets $\delta\cloud$ and $\mc T$ gets $\delta\target$ with $\delta\cloud\veebar\delta\target = (\alpha\leq\beta)$ \Comment $\veebar$ denotes the exclusive or operation
\State $\mc T$: encrypt $[[z\div 2^l]]$ and $[[\delta\target]]$ and send them to $\mc C$
\State $\mc C$: $[[t']] \leftarrow \begin{cases} [[\delta\target]] &\text{ if } \delta\cloud = 1 \\ [[1]]\oplus(-1)\otimes[[\delta\target]] &\text{ otherwise}  \end{cases}$
\State $\mc C$: $[[t]]\gets [[z\div 2^l]]\oplus(-1)\otimes ([[\rho\div 2^l]]\oplus [[t']])$ 
\Statex \Comment{$t \gets z\div 2^l - \rho\div 2^l - t'$}
\State $\mc C$: send $[[t]]$ to $\mc T$
\State $\mc T$: decrypts $[[t]]$
\end{algorithmic}
\end{protocol}
 
The comparison protocol that will be used in our optimization protocol is as follows. Let $\mc C$ have two encrypted values under the Paillier scheme $[[a]]$ and $[[b]]$ that it obtained after running Protocol~\ref{alg:rand_step}, and let $\mc T$ have the decryption key. Furthermore, let $\mc T$ also have the decryption key of the DGK homomorphic encryption scheme, which we describe in Appendix~\ref{app:comparison}. 
At the end of the protocol, $\mc T$ will have the result of the comparison in the form of one bit $t$ such that $(t=1)\Leftrightarrow (a\leq b)$. Let $l$ denote the number of bits of the unencrypted inputs $a,b$.
Protocol~\ref{alg:comp_DGKV} is based on the fact that the most significant bit of $(b-a+2^l)$ is the bit that indicates if $(a\leq b)$.

The security of this protocol is proved in \cite{DGK07,DGK09correction,Veugen12} and the proof is omitted here.
 
\begin{proposition}\label{prop:comp_DGK}
	Protocol~\ref{alg:comp_DGKV} is secure in the \mbox{semi-honest} model, according to Definition~\ref{def:view}.
\end{proposition}

\subsection{Secure update protocol}\label{subsec:update}
Moreover, we need to ensure that when the cloud $\mc C$ updates the value of the dual iterate at iteration $k+1$ in equation~\eqref{eq:projection}, it does not know the new value. The solution is to make the cloud blind the values of $[[a]]$ and $[[b]]$ and send them to the target node in this order, where the latter selects the value accordingly to the comparison result and then sends it back to the cloud. However, there are two important issues that have to be addressed in order for the update step to not leak information about the sign of the iterate: the blinding should be additive and effectuated with different random values, and the ciphertexts should be refreshed. The reasons are the following: if the blinding is multiplicative, by decrypting the product, the target knows which one of the values is zero. Moreover, if the two values are additively blinded with the same random value, the target can subtract them and reveal at least if the value is zero. Re-randomization of the encryptions is necessary so that the cloud cannot simply compare $[[a]]$ and $[[b]]$ with the received value. This can be done by adding an encryption of zero or by decryption followed by encryption. 
Protocol~\ref{alg:update_step} is the solution to the update problem:

\begin{protocol}[Secure update of the dual variable]
  \label{alg:update_step}
  \begin{algorithmic}[1]
\small
 \Require{$\mc C$: $[[a]],[[b]]$; $\mc T$: $t$ such that $(t=1)\Leftrightarrow(a\leq b)$}
 \Ensure{$\mc C$: $[[\mu]]$}
 \State $\mc C$: choose two random numbers $r,s$
 \State $\mc C$: $[[\bar a]] \leftarrow [[a]]\oplus [[r]], [[\bar b]] \leftarrow [[b]]\oplus [[s]]$
 \State $\mc C$: send $[[\bar a]]$ and $[[\bar b]]$ to $\mc T$
 \If{$t=0$} {~$\mc T$: $[[v]]\leftarrow \widetilde{[[\bar a]]} = [[\bar a]] + [[0]]$}\Else  {~$\mc T$: $[[v]]\leftarrow \widetilde{[[\bar b]]}=[[\bar b]] + [[0]]$}\EndIf \Comment Refresh the ciphertext
 \State $\mc T$: send $[[v]]$ and $[[t]]$ to $\mc C$
 \State $\mc C$: $[[\mu]]\gets [[v]]\oplus r\otimes[[t]]\oplus [[-r]] \oplus (-s)\otimes [[t]]$ 
 \Statex \Comment{$\mu\gets v + r(t-1) - st$}
\end{algorithmic}
\end{protocol}

The intuition behind the protocol is that if $a\leq b$, then $t=1$ and we obtain $\mu = \bar b - s = b$, and otherwise, $t=0$ and we obtain $\mu = \bar a - r = a$. In both cases the cloud correctly updates the dual variable with the projected value.

\subsection{Protocol for solving strictly-convex quadratic problems}\label{subsec:protocol}

Having defined these protocols, we can now build a protocol that represents one iteration \eqref{eq:projection} of the dual projected gradient ascent method. 

\begin{protocol}[Secure iteration of the dual projected gradient ascent method]
  \label{alg:iteration}
  \begin{algorithmic}[1]
\small
 \Require{$\mc C$: $A\cloud\in \mathbb R^{m\times n}, Q\cloud\in \mathbb S^n_{++}, [[b\agent]], [[c\agent]],\eta > 0, [[\mu_k]]$; $\mc T$: $sk\target  $}
 \Ensure{$\mc C$: $[[\mu_{k+1}]]$}
	\State $\mc C$: $[[\nabla g(\mu_k)]] \gets (-A\cloud Q\cloud^{-1} A\cloud^\intercal)\otimes[[\mu_k]] \oplus (-A\cloud Q\cloud^{-1})\otimes[[c\agent]]\oplus (-1)\otimes [[b\agent]]$ 
	\Statex \Comment Compute the encrypted gradient as in~\eqref{eq:grad}
	\State $\mc C$: $[[\bar \mu_k]] \gets [[\mu_k]] \oplus \eta\otimes [[\nabla g(\mu_k)]]$ \Comment Update the value in the ascent direction
	\State $\mc C, \mc T$ truncate $[[\bar \mu_k]]$ to $l$ bits
	\State $\mc C$ execute Protocol~\ref{alg:rand_step}: $\mc C$ gets $[[a_k]],[[b_k]]$ 
 \Comment Randomly assign $[[a_k]],[[b_k]]$ with values of $[[\bar \mu_k]], [[0]]$
	\State $\mc C, \mc T$ execute Protocol~\ref{alg:comp_DGKV} element-wise on inputs $[[a_k]],[[b_k]]$: $\mc T$ gets $t_k$ 
\Comment Secure comparison protocol
	\State $\mc C, \mc T$ execute Protocol~\ref{alg:update_step}: $\mc C$ obtains $[[\mu_{k+1}]]$ 
 \Comment Secure update protocol that ensures $\mu_{k+1} = \max\{\bar \mu_k, 0\}$
\end{algorithmic}
\end{protocol}

Line 3 ensures that the updated iterate has the required number of bits for the comparison protocol. This step is achieved by an exchange between the cloud and target node: the cloud additively blinds the iterate by a random number, sends it to the target node, which decrypts and truncates the sum and sends it back, where the cloud then subtracts the truncated random number.

The proof of security in the \mbox{semi-honest} model follows similar steps as in the argmax protocol in~\cite{Bost15} and we will address it in Appendix~\ref{subsec:proof_thm1}. 

\begin{proposition}\label{prop:iteration}
	Protocol~\ref{alg:iteration} is secure in the \mbox{semi-honest} model, according to Definition~\ref{def:view}.
\end{proposition}

Using the building blocks described above, we can finally assemble the protocol that privately solves the constrained quadratic optimization problem \eqref{eq:problem} with private data and sends the optimal solution to the target node. The public key $pk\target$ and bit-length $l$ are known by all the parties, hence we omit them from the inputs.

\begin{protocol}[Privacy preserving algorithm for solving strictly-convex quadratic optimization problems]
  \label{alg:main_alg}
  \begin{algorithmic}[1]
\small
 \Require{$\mc A_{i=1,\ldots,p}$: $b\agent= \{b_j\}_{j=1,\ldots,m}, c\agent= \{c_j\}_{j=1,\ldots,n}$; $\mc C$: $A\cloud\in \mathbb R^{m\times n}, Q\cloud\in \mathbb S^n_{++},K$; $\mc T$: $sk\target,K $}
 \Ensure{$\mc T$: $x^\ast$}
 \For{i=1,\ldots,p}
	\State $\mc A_i:$ encrypt the private information $msg_i \gets$ $ ([[b_i]]$, $[[c_i]])$
	 \State $\mc A_i:$ send the encrypted messages to $\mc C$
 \EndFor
 \State $\mc C$: Construct the vectors $[[b\agent]]$ and $[[c\agent]]$ from the messages
 \State $\mc C$: $\eta \gets 1/\lambda_{max}(A\cloud Q\cloud^{-1}A\cloud^\intercal)$
 \State $\mc C$: Choose a random positive initial value $\mu_0$ for the dual variable and encrypt it: $[[\mu_0]]$
\For {each $k=0,\ldots,K-1$}
	\State $\mc C,\mc T$ execute Protocol~\ref{alg:iteration}: $\mc C$ gets $[[\mu_{k+1}]]$ 
\Comment{$\mc C,\mc T$ securely effectuate an iteration of the dual projected gradient ascent}
\EndFor
	\State $\mc C$: $[[x^\ast]] \gets (-Q\cloud^{-1} A\cloud ^\intercal)\otimes[[\mu_K]] \oplus (-Q\cloud^{-1})\otimes[[c\agent]]$ and send it to $\mc T$ \Comment Compute the primal optimum from the optimal dual solution as in~\eqref{eq:KKT1}
	\State $\mc T$: Decrypt $[[x^\ast]]$ and output $x^\ast$
\end{algorithmic}
\end{protocol}

\subsection{Fixed-point arithmetic}~\label{subsec:quantization}
The optimization problem~\eqref{eq:problem} is defined on real variables, whereas the Paillier encryption scheme is defined on integers. To address this issue, we adopt a \mbox{fixed-point} arithmetic setting, where we allow for a number to have a fixed number of fractional bits. First, we consider numbers that have the magnitude between $-2^{l_i-1}< x < 2^{l_i-1}$. Second, we consider a value having $l_i$ bits for the integer part and $l_f$ bits for the fractional part. Therefore, by multiplying the real values by $2^{l_f}$ and truncating the result, we obtain integers. We choose $l=l_i+l_f$ large enough such that the loss in accuracy is negligible and assume that there is no overflow. For ease of exposition, we consider this data processing done implicitly in the protocols described. 

The random numbers used for blinding the sensitive values (in Protocols~\ref{alg:comp_DGKV} and~\ref{alg:update_step}) are sampled uniformly from the integers in $\mathbb Z_N$ of $l+\lambda$ bits, where $\lambda$ is the statistical security parameter, as already explained in Section~\ref{sec:problem_setup}, chosen such that brute-forcing the solution is intractable. In order to guarantee correctness of the comparison protocol, no overflow must take place, so we must impose $log_2 N > l+ \lambda + 1$.

The errors in the solution caused by the fixed-point arithmetic operations necessary for the encryption can be analyzed with the same tools as in~\cite{Jerez14,Rubagotti2016real,Alexandru18}. The round-off errors can be regarded as states in a stable dynamical system with bounded disturbances, and hence, have a bounded norm that offers a guide on how to choose the number of fractional bits $l_f$ for the fixed-point representation. On the other hand, the overflow and quantization errors depend on the magnitude of the dual iterates. We considered feasible and bounded problems -- the dual problem~\eqref{eq:dual} has a finite solution -- therefore, one can select the number of integer bits $l_i$ in the representation such that no overflow occurs.


\section{Privacy of Quadratic Optimization Protocol}\label{sec:security_proof}

We will now introduce the main theoretical results of the paper. In the interest of completeness of the paper, more theoretical preliminaries needed for the privacy results, such as semantic security, are given in Appendix~\ref{app:cryptographic_prel}. We provide the statements of the results and the main ideas of the proofs here, and include the detailed arguments in Appendix~\ref{app:proof}. 

\begin{theorem}\label{thm:main_alg}
	Protocol~\ref{alg:main_alg} achieves privacy with respect to Definition~\ref{def:view} for non-colluding parties.
\end{theorem}

The intuition for the proof is as follows. Consider an iteration of the gradient ascent in Protocol~\ref{alg:main_alg}. 
Firstly, in the Paillier cryptosystem, two ciphertexts are computationally indistinguishable to a party that does not have access to the decryption key. 
Secondly, the exchanges between the cloud and the target are additively blinded using a different random number uniformly sampled from a large enough range (at least $\lambda$ bits more over the values that are desired to be blinded, where the size of $\lambda$ is chosen appropriately, as discussed in Section~\ref{subsec:symmetric}. This means that the blinded message is statistically indistinguishable from a random number sampled from the same distribution. Thirdly, the ciphertexts are refreshed (a different encryption of the same value), after each exchange, so a party that does not have access to the decryption key cannot infer information about the encrypted values by simply comparing the ciphertexts. Then, none of the parties can infer the magnitude or the sign of the private variables. However, we need to show that privacy is not broken by running an iteration multiple times. We prove that storing the exchanged messages does not give any new information on the private data, using similar arguments. Formally, using Definition~\ref{def:view}, we construct a probabilistic polynomial-time simulator that randomly generates messages from the inputs and outputs such that its view and the view of the adversary, on the same inputs, are computationally indistinguishable. 
The correctness of Protocol~\ref{alg:main_alg} is immediate and follows from the correctness of the dual gradient ascent algorithm and the correctness of the comparison protocol. The detailed proof is given in Appendix~\ref{subsec:proof_thm1}.

Let us now consider collusions between the parties in the setup and prove privacy of Protocol~\ref{alg:main_alg} under coalitions. The definition of privacy is naturally extended in this case and we can further establish the computational security in the following result -- see Appendix~\ref{app:cryptographic_prel} for more details. Definition~\ref{def:coalition_view} states that even under collusions, the protocol securely computes a functionality, which in this case, is solving a quadratic optimization problem. No further information is revealed than what can inferred from the coalition's inputs and outputs. However, if all agents and the cloud collude then they have access to all the information to solve the problem, in which case the above result is rather vacuous. Similarly, if the cloud colludes with the target node, then it can gain access to all the private data of the agents. Hence, we consider coalitions that involve either all the agents or the cloud and the target node to be prohibited (and unrealistic) and only consider coalitions between a strict subset of the agents and the cloud or between a strict subset of the agents and the target node. We give a complementary analysis on the information gained from the inputs and outputs of coalitions in Section~\ref{sec:discussion}.

\begin{theorem}\label{thm:main_alg_coal}
	Protocol~\ref{alg:main_alg} achieves privacy with respect to Definition~\ref{def:coalition_view} against coalitions.
\end{theorem}

The proof of Theorem~\ref{thm:main_alg_coal} can be derived from the proof of Theorem~\ref{thm:main_alg} because the agents only contribute with their inputs to the view of the coalition, and not with new messages than. The proof is given in Appendix~\ref{subsec:proof_thm2}.
%


\section{Alternative quadratic optimization protocol without secure comparison}\label{sec:relaxation_security}
As explained in Section~\ref{subsec:comparison}, the major computation and communication overhead of the above protocol is the secure comparison protocol, required to project dual variables to non-negative numbers. In this section, we describe a computationally less involved alternative approach, which bypasses the need for the secure comparison protocol, at the expense of revealing more information. This approach is developed in more detail in previous work~\cite{Shoukry16}.

Specifically, consider a step of the dual gradient algorithm where the cloud maintains an unprojected gradient ascent step $\bar{\mu}_k = \mu_k - \eta [A\cloud Q\cloud^{-1} (A\cloud^\intercal\mu_k + c\agent) + b\agent]$ encrypted using the public key of the target node. Suppose the cloud multiplies the elements in this vector with random scalar values $r_k$ uniformly distributed over the positive integers and sends the products to the target node. The latter can decrypt the message using its private key and gain access not to the actual unprojected iterate but to the randomly scaled version of it, which reveals the sign. It can then project it to the non-negative orthant in an unencrypted fashion $\max\{0, (r_k)_i (\bar{\mu}_k)_i \}$ for $i=1,\ldots,m$, and finally encrypt the result using its own public key and return it to the cloud. The cloud can divide the result with the previously selected values $r_k$ to compute in an encrypted fashion the elements $[[(\mu_{k+1})_i]] = (1/(r_k)_i)\otimes [[\max\{0, (r_k)_i (\bar{\mu}_k)_i \}]]$ which is equivalent to the encrypted actual projected step $[[\max\{0, (\bar{\mu}_k)_i \}]]$ for $i=1,\ldots,m$, because division with a positive number commutes with the max operator. 

This protocol is presented next, and it can be employed as an alternative for the more demanding Protocol~\ref{alg:iteration}, as it bypasses the complexity of the secure comparison part. On the other hand, as can be seen from the outputs of the protocol, it reveals more information to the target node than the secure comparison approach. In particular it reveals a scaled version of the unprojected dual variable which in turn does not reveal the magnitude of this value but reveals whether this is positive, negative, or zero.

\begin{protocol}[Alternative iteration of the dual projected gradient ascent method]
  \label{alg:iteration_alternative}
  \begin{algorithmic}[1]
\small
 \Require{$\mc C$: $A\cloud\in \mathbb R^{m\times n}, Q\cloud\in \mathbb S^n_{++}, [[b\agent]], [[c\agent]],\eta > 0, [[\mu_k]]$; $\mc T$: $sk\target  $}
 \Ensure{$\mc C$: $[[\mu_{k+1}]]$; $\mc T$: $(r_k)_i (\bar\mu_k)_i$, $i=1,\ldots,m$}
	\State $\mc C$: $[[\nabla g(\mu_k)]] \gets (-A\cloud Q\cloud^{-1} A\cloud^\intercal)\otimes[[\mu_k]] \oplus (-A\cloud Q\cloud^{-1})\otimes[[c\agent]]\oplus (-1)\otimes [[b\agent]]$ 
	\Comment Compute the encrypted gradient as in~\eqref{eq:grad}
	\State $\mc C$: $[[\bar \mu_k]] \gets [[\mu_k]] \oplus \eta\otimes [[\nabla g(\mu_k)]]$ \Comment Update the value in the ascent direction
	\State $\mc C$: generate a random uniform positive scalar $r_k$ and send $(r_k)_i [[(\bar \mu_k)_i]]$ to $\mc T$ 
	\State $\mc T$: decrypt and truncate $(r_k)_i \otimes [[(\bar \mu_k)_i]]$, $i=1,\ldots,m$
	\State $\mc T$: compute $ \max\{ (r_k)_i (\bar \mu_k)_i, 0\}$
	\State $\mc T$: encrypt $ [[\max\{ (r_k)_i (\bar \mu_k)_i, 0\}]]$ and send to $\mc T$ 
	\State $\mc C$: compute $[[\mu_{k+1}]] = (1/(r_k)_i) \otimes [[\max\{ (r_k)_i (\bar \mu_k)_i, 0\}]]$, $i=1,\ldots,m$  
\end{algorithmic}
\end{protocol}

Recall that we can only perform operations on integers when using the Paillier cryptosystem. To this end, we performed an implicit multiplication by $2^{l_f}$ and truncation when encrypting numbers and a division by the corresponding factor when decrypting. For Algorithm~\ref{alg:iteration_alternative} to work, we slightly modify this technique. The reasoning is the following: we want the values of $r_k$ to be large enough to blind the magnitude of the elements of $\bar \mu_k$, so we will randomly sample $r_k$ from the values of $\gamma+l$ bits, where $\gamma$ is the security size for the multiplicative blinding. Furthermore, we need to be able to represent $1/(r_k)_i$ on the message space and to have enough precision as to obtain the correct value of $(\mu_{k+1})_i$ in line~7. Hence, we will perform implicit multiplications and divisions by $2^{l_f'}$, for $l_f'>l + \gamma$. 
This change shows that, although significantly less communication is required to perform the projection operation, all the other operations are more expensive, because we work with larger numbers. 

\begin{theorem}\label{thm:alternative}
Protocol~\ref{alg:main_alg} with the iterations as in Protocol~\ref{alg:iteration_alternative} achieves privacy according to Definitions~\ref{def:view} and~\ref{def:coalition_view}.
\end{theorem}

The proof follows similar steps as in the proof of the main protocol in Section~\ref{sec:security_proof} and is sketched in Appendix~\ref{subsec:proof_thm3}.
%


\section{Privacy discussion}\label{sec:discussion}

In this section we discuss privacy aspects that differ from the computational notions of Definitions~\ref{def:view} and~\ref{def:coalition_view}. Even if a protocol does not leak any information, the known information in a coalition (e.g., the output and some of the inputs) can be used by the coalition to infer the rest of the private inputs. More specifically, in our optimization problem, the private variables and the optimal solution are coupled via the optimality conditions \eqref{eq:KKT1}-\eqref{eq:KKT4}, which are public knowledge, irrespective of the protocol, and may be used for the purpose described above. 

Consider the following definition that concerns the retrieval of private data from adversarial/known data.

\begin{definition}\label{def:non-unique}(Non-unique retrieval)
Let $p$ be the private inputs of a problem and let an algorithm $A(p)$ solve that problem. Let $\mc K$ be the adversarial knowledge of the problem, which can contain public information, some private information (including some parts of the input) and the output of algorithm~$A$ for the adversary, denoted by~$A^{\mc K}(p)$. We say $p$ cannot be uniquely retrieved by the adversary if there exists a set $\mc U$, such that $p\in\mc U$, $|\mc U|\geq 2$ and:
\[\forall p'\in\mc U: A^{\mc K}(p) = A^{\mc K}(p').\]
\end{definition}
Definition~\ref{def:non-unique} can be modified to be stronger by requiring the set $\mc U$ to have an infinite number of elements.

In what follows, beyond the computational security analysis of the previous sections, we carry out an algebraic analysis on a \mbox{black-box} protocol that given the agent's private data $b\agent,c\agent$ and the cloud's matrices $Q\cloud, A\cloud$, outputs the solution $x^\ast$ of Problem~\eqref{eq:problem} to the target node. We provide conditions such that a coalition cannot uniquely determine unknown private inputs from the output and a set of inputs, in the sense of Definition~\ref{def:non-unique}. In particular, this analysis applies to Protocol~\ref{alg:main_alg} which, assuming it runs for sufficient iterations, outputs the desired result $x^\ast$ to the target node. We perform this algebraic analysis in the space of real numbers, which can be further expanded to fixed-point arithmetics for large enough precision. 

Suppose without loss of generality that a coalition between $\bar p$ agents ($1\leq \bar p < p$) has access to the elements $b_1, \ldots, b_{\bar{m}}$ with $0 \leq \bar{m} \leq m$, and $c_1, \ldots, c_{\bar{n}}$ with $0 \leq \bar{n} \leq n$. Then let us define the decomposition of the matrix $A\cloud$ as:
\begin{equation}
	A\cloud = \left[ \begin{array}{c}
	A_1 \\ \hline \begin{array}{c|c} A_{21} & A_{22} \end{array}
	\end{array}
	 \right]
\end{equation}
where the matrices $A_1 \in \mathbb{R}^{\bar{m} \times n}$, $A_{21} \in \mathbb{R}^{(m - \bar{m}) \times \bar{n}}$, \mbox{$A_{22} \in \mathbb{R}^{(m -\bar{m}) \times (n -\bar{n})}$}.

\begin{proposition}\label{prop:assum_1} Consider a protocol solving Problem~\eqref{eq:problem} and a coalition between the target node and agents with access to $\bar{m}$ of the values of $b\agent$ and $\bar{n}$ of the values of $c\agent$. Suppose the cost and constraint matrices $A\cloud$, $Q\cloud$ are public. Then: 
	\begin{enumerate}[leftmargin=10pt]
		\item[(1)] if $\bar{m}<m$ and there exists a vector $\delta \in \mathbb R^{m-\bar m}$ such that $ \delta\neq 0, \delta \succeq 0$ and $ A_{21}^\intercal \delta = 0$, then the coalition cannot uniquely retrieve the value of $b\agent$;
		\item[(2)] if additionally $\bar{n}<n$ and $A_{22}^\intercal \delta \neq 0$ then the coalition cannot uniquely retrieve the value of $c\agent$.
	\end{enumerate}
\end{proposition}

The proof is based on the fact that the variables $b\agent$, $c\agent$, $x^\ast$, $\mu^\ast$ satisfy the optimality conditions \eqref{eq:KKT1}-\eqref{eq:KKT4}. Specifically, these are conditions that the unknown variables $b_{\bar{m}+1}, \ldots, b_m$ and $c_{\bar{n}+1}, \ldots, c_n$, as well as the optimal dual variables $\mu^\ast$, must satisfy given all the rest of the known variables. Hence, if there are multiple such solutions to the KKT conditions the coalition cannot uniquely determine the private variables. The proof is given in Appendix~\ref{subsec:proof_assum_1}.

\begin{proposition}\label{prop:assum_2}
Consider a protocol solving Problem~\eqref{eq:problem} and a coalition between the cloud and agents with access to $\bar{m}$ of the values of $b\agent$ and $\bar{n}$ of the values of $c\agent$. Then, a coalition that satisfies $\bar m < m$ and $\bar n < n$ cannot uniquely retrieve the values of $b\agent,c\agent$ and $x^\ast$.
\end{proposition}

The proof easily follows from the fact that $Q\cloud$ is a positive definite matrix, hence, the solution $x^\ast$ is finite, and $A\cloud$ does not have any columns or rows of zeros. A coalition between the cloud and the $\bar p<p$ agents cannot solve~\eqref{eq:problem} since it lacks all the data to define it ($\bar m < m$ and $\bar n < n$), so it cannot uniquely retrieve $x^\ast$ and the rest of the agents' private data.

Propositions~\ref{prop:assum_1} and~\ref{prop:assum_2} give general sufficient conditions for the desired non-unique retrieval property to hold and are independent of the exact problem instance values. If the conditions stated are not satisfied, then the previous result in Theorem~\ref{thm:main_alg_coal} still holds. However, the additional non-unique retrieval property may fail because the inputs and outputs of the coalition are sufficient to leak information about the rest of the private inputs. 
The above analysis can also be extended to the case where some of the matrices $Q_{\cloud}, A_{\cloud}$ are private information.

Let us now use Definition~\ref{def:non-unique} to analyze Protocol~\ref{alg:iteration_alternative} and the effect the release of more information to the target node has on the privacy of the inputs of the honest agents. We perform the analysis at a setup where the protocol runs for a sufficient number of iterations and hence the dual variable $\mu_k$ has converged to the true optimal value $\mu^\ast$ and the algorithm has also converged to the true optimal primal value $x^\ast$. 
Suppose the target node has access to the sign of the unprojected optimal dual variable $\mu^\ast + \nabla g(\mu^\ast)$ -- note that this is the case when we employ Protocol~\ref{alg:iteration_alternative}. In combination with the solution $x^\ast$, this information can be further employed by the target node to infer private information of the agents.

\begin{proposition}\label{prop:b_distinguishability}
Consider a protocol solving Problem~\eqref{eq:problem} where the target node has access to the solution $x^\ast$ and also the sign of the unprojected optimal dual variables $t := \text{sign}(\mu^\ast + \nabla g(\mu^\ast)) \in \{+1, -1, 0\}^m$. Suppose further the matrix $A\cloud$ is publicly available. Then the private values $b\agent$ cannot be uniquely retrieved by the target node if and only if $t_i < 0$ for some $i\in \{1,\ldots,m\}$.
\end{proposition}

Here $t_i < 0$ implies that the corresponding optimal dual value is zero, $\mu^\ast_i = 0$. This means that the corresponding constraint in problem \eqref{eq:problem} is inactive at the optimal solution $x^\ast$~\cite[Ch. 5]{BV04convex} and the target node cannot uniquely determine the $i$th element of the corresponding bound $b\agent$. In the opposite case, if all constraints \eqref{eq:KKT23} are either active or redundant at the optimal solution ($\mu^\ast\succeq 0$), this is revealed to the target node because in that case $t_i \geq0$, and the private value $b\agent$ is uniquely determined by $b\agent = A\cloud x^\ast$.

\begin{proposition}\label{prop:c_distinguishability}
Consider the setup of Proposition~\ref{prop:b_distinguishability} and the matrix $Q\cloud$ is publicly available. The private values $c\agent$ cannot be uniquely retrieved from the outputs $x^\ast$ and $t$ of the target node if and only if $t_i > 0$ for some $i\in \{1,\ldots,m\}$.
\end{proposition}

The case $t_i> 0$ corresponds now to the case where the corresponding constraint is active at the optimal solution $x^\ast$. When this fails, all constraints are inactive at $x^\ast$ and they do not play a role. Hence, we have an unconstrained quadratic problem in \eqref{eq:problem}, i.e., the optimal solution satisfies the first order condition $Q\cloud x^\ast + c\agent = 0$, which reveals the value of $c\agent$ to the target. To guarantee privacy with respect to both private values $b\agent, c\agent$ we need both an inactive constraint ($t_i <  0$) as well as an active constraint ($t_j >  0$) for some $i,j \in \{1, \ldots, m \}$, leaving some uncertainty in the estimations performed by the target node. Similar analysis may be performed for collusions, e.g., of the target node and some agents. 

Finally, we note that here we analyzed the non-unique retrieval properties of the problem considering black-box protocols. It is of value to also consider additional steps when designing a protocol that solves Problem~\eqref{eq:problem} to prevent information leakage if the conditions in the propositions are not satisfied. It may be possible to use techniques such as Differential Privacy to augment the cryptographic privacy of the protocol with output privacy, by adding correlated data to the inputs. This approach lies out of the scope of our paper, but works such as~\cite{Huang15,Han17,Nozari18} may be helpful to this end.
%


\section{Implementation}\label{sec:complexity_time}

The efficiency of a secure multi-party protocol is measured in the complexity of the computations performed by each party, along with the rounds of communications between the parties. While the former is relevant from the perspective of the level of computational power required, the latter are relevant when communication is slow. 

Let $\sigma$ be the bit-length of the modulus $N$. Then, the size of a Paillier ciphertext is $2\sigma$, regardless of the size of the plaintext in $\mathbb Z_N$. 
A Paillier encryption takes one online exponentiation and one multiplication modulo $N^2$, and a decryption takes one exponentiation modulo $N^2$ and one multiplication modulo $N$. The addition of two ciphertexts takes one modular multiplication modulo $N^2$. A multiplication of a ciphertext with a plaintext of $l$ bits is achieved as a modular exponentiation modulo $N^2$. A multiplication of two elements in $\mathbb Z^\ast_{N^2}$ can be achieved in $O((2\sigma)^{1.7})$. A modular exponentiation in $N^2$ with an $l$-bit exponent can be computed in $lO((2\sigma)^2)$ and can be sped up via the Chinese Remainder Theorem when the factorization of $N$ is known.
A DGK encryption takes one modular exponentiation and one modular multiplication in $\mathbb Z^\ast_{N_{DGK}}$, and a DGK decryption -- checking if the encrypted message is 0 or not -- takes an exponentiation modulo $p_{DGK}$.

In the setup we considered, the agents are low-power platforms and they are only required to effectuate one encryption and send one message, but the cloud and the target node are platforms with higher computational capabilities. Protocol~\ref{alg:main_alg} involves $\mc O(K)$ matrix-encrypted vector multiplications and $\mc O(Km)$ encrypted scalar additions performed at the cloud, $\mc O(Kml)$ DGK decryptions, $\mc O(Km)$ Paillier decryptions and $\mc O(Km)$ encrypted scalar additions and scalar-encrypted scalar multiplications at the target node, with $\mc O(K)$ batches of messages communicated between the two. The random numbers and their encryptions can be precomputed.

We implemented the protocol proposed in Section~\ref{subsec:protocol} in Python 3 and ran it on a 2.2 GHz Intel Core i7 processor. For random instances of the data in Problem~\eqref{eq:problem}, with $\sigma=1024$ and $l=32$ bits, and 0, 10 and 20 ms delay for communication, we obtain the average running times depicted in Figure~\ref{fig:results_delay} and Figure~\ref{fig:results_delay_alt}. The simulations are run for 30 iterations, for ease of comparison. 

Figure~\ref{fig:results_delay} depicts the online execution time of Protocol~\ref{alg:main_alg} with the iterations as in Protocol~\ref{alg:iteration}. As expected, because we work in the dual space, the time required for running the iterations varies very little with the number of variables, but increases with the number of constraints. 
However, in the case of Figure~\ref{fig:results_delay_alt}, which depicts the online execution time of Protocol~\ref{alg:main_alg} with the iterations as in Protocol~\ref{alg:iteration_alternative}, the time varies more with the number of variables $n$. The reason is that we work with substantially larger numbers than in the previous case, due to the large factor with which we have to multiply in order to be able to de-mask the dual iterate, which amplifies the time difference between problems of different dimensions.

The trade-off between the privacy and communication can be seen when we artificially add a communication delay of 10 ms, respectively 20 ms between the cloud and the target node, to simulate the delays that can occur on a communication network. It can be observed in Figures~\ref{fig:results_delay} and~\ref{fig:results_delay_alt} that the communication delay has a much smaller effect on the less private Protocol~\ref{alg:iteration_alternative}, which does not use the secure comparison protocol, than on the fully private Protocol~\ref{alg:iteration}. 

\begin{figure}
  \centering
    \includegraphics[width=0.45\textwidth]{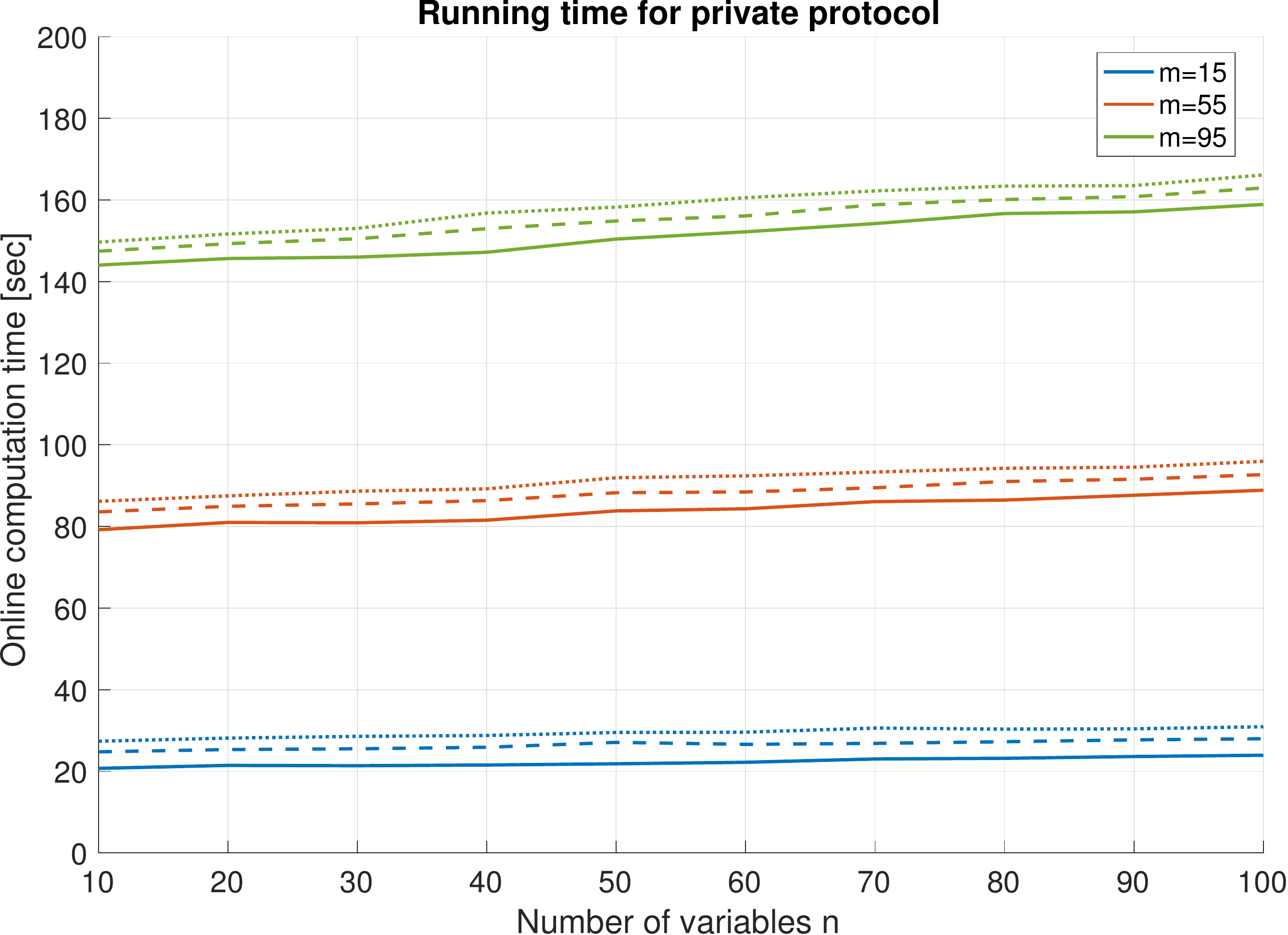}
  \caption{Comparison between the average running time of Protocol~\ref{alg:main_alg} with the iterations as in Protocol~\ref{alg:iteration} for problem instances with the number of variables on the abscissa and the number of constraints in the legend. The full lines correspond to the running times of the protocol with no communication delay, the dashed lines correspond to a 10 ms delay and the dotted lines to a 20 ms communication delay. The simulation is run for 30 iterations, a 1024 bit key and 32 bit messages, with 16 bit precision. The statistical parameter for additive blinding is 100 bits.}
   \label{fig:results_delay}
\end{figure}

\begin{figure}
  \centering
    \includegraphics[width=0.45\textwidth]{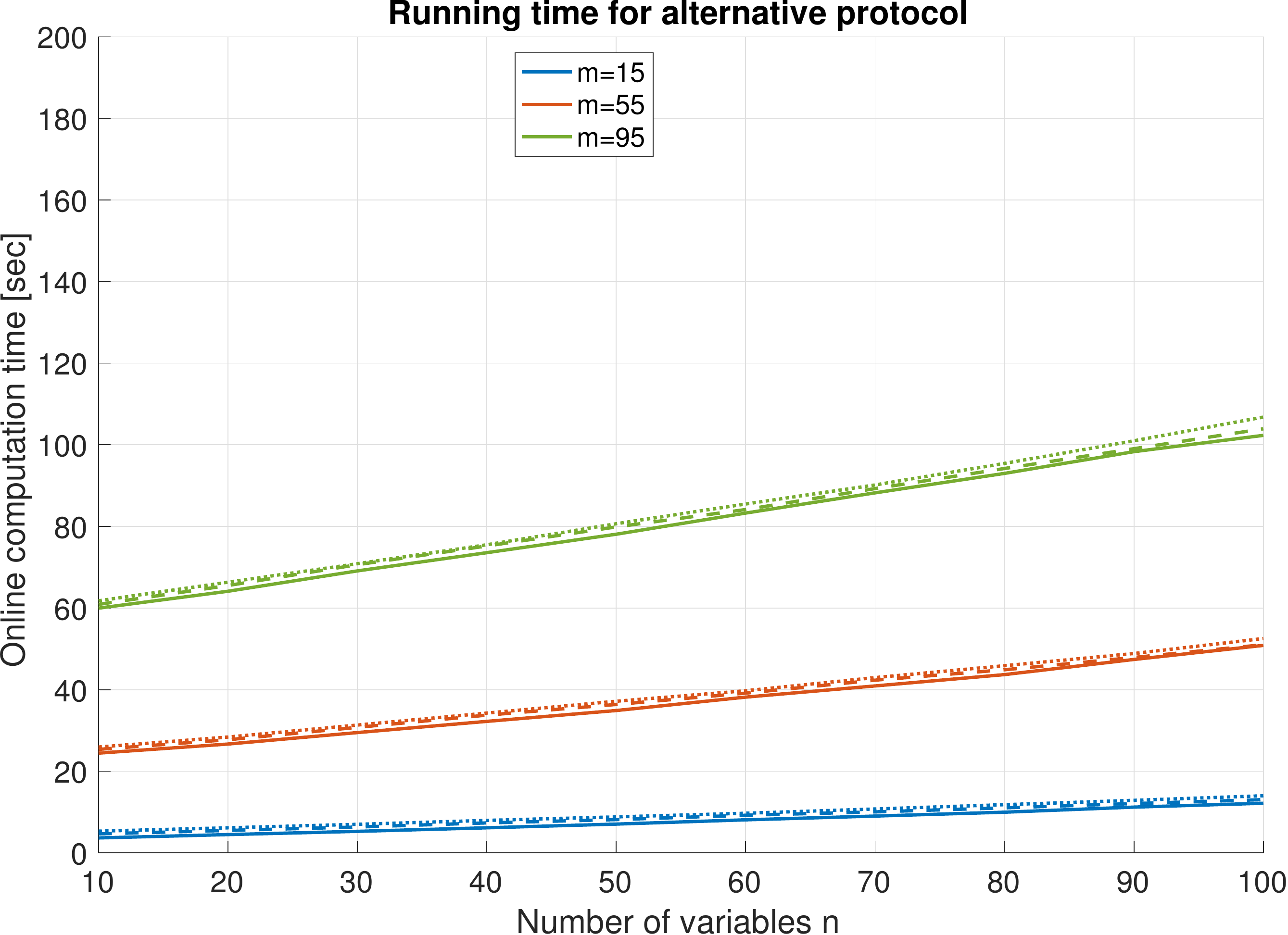}
  \caption{Comparison between the average running time of Protocol~\ref{alg:main_alg} with the iterations as in Protocol~\ref{alg:iteration_alternative} for problem instances with the number of variables on the abscissa and the number of constraints in the legend. The full lines correspond to the running times of the protocol with no communication delay, the dashed lines correspond to a 10 ms delay and the dotted lines to a 20 ms communication delay. The simulation is run for 30 iterations, a 1024 bit key and 32 bit messages, with 16 bit precision. The statistical parameter for multiplicative blinding is 40~bits.}
   \label{fig:results_delay_alt}
\end{figure}
%

\section*{Acknowledgment}

This work was partially sponsored by the NSF awards 1739816, 1705135, CNS-1505799 and the Intel-NSF Partnership for Cyber-Physical Security and Privacy, ONR N00014-17-1-2012, and by TerraSwarm, one of six centers of STARnet, a Semiconductor Research Corporation program sponsored by MARCO and DARPA. The U.S. Government is authorized to reproduce and distribute reprints for Governmental purposes notwithstanding any copyright notation thereon. The views and conclusions contained herein are those of the authors and should not be interpreted as necessarily representing the official policies or endorsements, either expressed or implied, of NSF, or the U.S. Government. Andreea Alexandru would like to thank Brett Hemenway for helpful discussions.


\ifCLASSOPTIONcaptionsoff
  \newpage
\fi

\bibliographystyle{IEEEtran}
\bibliography{IEEEabrv,biblo}


\appendices

\section{Cryptographic tools}\label{app:cryptographic_prel}

\allowdisplaybreaks
\setcounter{theorem}{0}
     \renewcommand{\thetheorem}{\Alph{section}.\arabic{theorem}}
\setcounter{definition}{0}
     \renewcommand{\thedefinition}{\Alph{section}.\arabic{definition}}
\setcounter{lemma}{0}
     \renewcommand{\thelemma}{\Alph{section}.\arabic{lemma}}
\setcounter{proposition}{0}
     \renewcommand{\theproposition}{\Alph{section}.\arabic{proposition}}
\setcounter{remark}{0}
     \renewcommand{\theremark}{\Alph{section}.\arabic{remark}}     
\setcounter{corollary}{0}
     \renewcommand{\thecorollary}{\Alph{section}.\arabic{corollary}}    

Proving a protocol is secure in the semi-honest model involves a heavy use of tools from cryptography, which we present here for the interested reader. 
In this Appendix, we aim to give closure to the security notions introduced in Section~\ref{sec:privacy_goals}, as well as provide the details of the Paillier encryptions scheme, used throughout the paper and generically introduced in Section~\ref{sec:encryption_sch}. Furthermore, we detail the tools used in the proof of Theorem~\ref{thm:main_alg}, such as semantic security.

\subsection{Security of multi-party protocols in the semi-honest model}\label{subsec:appendix_security}

Definition~\ref{def:view} and Definition~\ref{def:coalition_view} state that the view of any of the parties participating in the protocol, on each possible set of inputs, can be simulated based only on its own input and output. 
The intuition behind these definitions can also be viewed as follows: consider a \mbox{black-box} protocol that receives the inputs of Problem~\eqref{eq:problem} and solves it, not exchanging any messages in order to achieve the solution. Such a \mbox{black-box} protocol satisfies Definitions~\ref{def:view} and~\ref{def:coalition_view}. 
In other words, meeting the requirements of the above definitions is equivalent to not revealing any other information than what is already known to any of the parties, i.e. inputs, prescribed outputs, if any, and previously known side information, meaning details about the optimization algorithm. 

\begin{remark}\label{remark:aux}(\cite[Ch.~7]{Goldreich04foundations||},\cite[Ch.~4]{Goldreich03foundations|})
	\textbf{Auxiliary inputs}, which are inputs that capture additional information available to each of the parties (e.g. local configurations, side-information), are implicit in Definition~\ref{def:view} and Definition~\ref{def:coalition_view}. 
\end{remark}

\begin{remark}~\label{remark:leak}
Satisfying Definitions~\ref{def:view} and~\ref{def:coalition_view} is a stronger requirement than guaranteeing that an adversary cannot uniquely retrieve the data of the honest parties, i.e. Definition~\ref{def:non-unique}. 
\end{remark}

Revealing sensitive information does not always lead to a unique retrieval of the private data. Nevertheless, any piece of information revealed by the execution of the protocol, that cannot be obtained only from its inputs and outputs, leads to the violation of Definitions~\ref{def:view},~\ref{def:coalition_view}, even if the private data cannot be singled out with this information.

\subsection{Paillier cryptosystem}\label{subsec:appendix_paillier}
In this part, we will provide the rest of the details regarding the implementation of the Paillier cryptosystem, introduced in Section~\ref{sec:encryption_sch}. The pair of keys corresponding to this cryptosystem is $(pk,sk)$, where the public key is $pk = (N,g)$ and the secret key is $sk = (\gamma,\delta)$. $N$ is called the modulus and is the product of two large prime numbers $p,q$, and $g$ is an element of order $N$ in the multiplicative group $\mathbb Z^\ast_{N^2}$. Furthermore, \[\gamma = lcm(p-1,q-1),~\delta = ((g^\gamma~\text{mod}~N^2 -1)/N)^{-1}~\text{mod}~N.\]
A simpler variant to generate the pair of keys when $p$ and $q$ have the same number of bits is to choose $g = N+1$, $\gamma = \phi(N) = (p-1)(q-1)$, where $\phi$ is Euler's totient function and $\delta = \phi(N)^{-1}\bmod N$.

For a plaintext $a\in \mathbb Z_N$, the Paillier encryption is:
\begin{equation}
[[a]] = g^ar^N\bmod N^2,
\end{equation}
 where $r$ is a random non-zero element in $\mathbb Z_N$, which makes Paillier a probabilistic encryption scheme. Therefore, the ciphertexts do not preserve the order relation between the plaintexts. The decryption primitive is the following: 
\begin{equation}
a = (([[a]]^\gamma\bmod N^2-1)/N)\delta\bmod N,
\end{equation}
which uses the fact that $(1+N)^a = 1+Na \bmod N^2$. 

In the Paillier scheme, in order to obtain addition between plaintexts, the operation between ciphertexts is modular multiplication, which was denoted by $\oplus$ in the text: 
\begin{align}
\begin{split}
	[[a]]\cdot[[b]] &= g^ar^N\bmod N^2 \cdot g^br'^N\bmod N^2 \\ &= g^{a+b}(rr')^N\bmod N^2 \\ &= [[a + b]]\bmod N^2.
\end{split}
\end{align}

Negation is achieved by modular inverse:
\begin{align}
\begin{split}
[[a]]^{-1} &= g^{-a}(r^{-1})^N\bmod N^2 \\ &= [[-a]]\bmod N^2.
\end{split}
\end{align}

The multiplication between a plaintext value $c$ and an encrypted value $[[a]]$, which was denoted by $c\otimes[[a]]$ in the text, is obtained in the following way:
\begin{equation}
[[a]]^{c} = g^{ac}(r^c)^N\bmod N^2 = [[ca]].
\end{equation}

\begin{remark}
	The Paillier cryptosystem is malleable, which means that a party that does not have the private key can alter a ciphertext such that another valid ciphertext is obtained. Malleability is a desirable property in order to achieve cloud-outsourced computation on encrypted data, but allows ciphertext attacks. 
In this work, we assume that the parties have access to authenticated channels, therefore an adversary cannot alter the messages sent by the honest parties.
\end{remark}

\subsection{Semantic security of a cryptosystem}

Proving security in the \mbox{semi-honest} model involves the concepts of semantic security of an encryption scheme. Under the assumptions of decisional composite residuosity~\cite{Paillier99} and hardness of factoring, the partially homomorphic encryption schemes Paillier and DGK are semantically secure and have indistinguishable encryptions, which, in essence, means that an adversary that has the plaintext messages $a$ and $b$ cannot distinguish between the encryption $[[a]]$ and encryption $[[b]]$. 

\begin{definition}\label{def:semantic}(Semantic Security~\cite[Ch.~5]{Goldreich04foundations||})
An encryption scheme is \textbf{semantically secure} if for every probabilistic polynomial-time algorithm, A, there exists a probabilistic polynomial-time algorithm $A'$ such that for every two polynomially bounded functions $f,h:\{0,1\}^\ast\rightarrow \{0,1\}^\ast$ and for any probability ensemble $\{X_n\}_{n\in\mathbb N}$, $|X_n|\leq poly(n)$, for any positive polynomial $p$ and sufficiently large $n$: 
\begin{align*}
\text{Pr}& \left[ A(E(X_n),h(X_n),1^n) = f(X_n)\right] < \\ \text{Pr}& \left[ A'(h(X_n),1^n) = f(X_n)\right] + 1/p(n),
\end{align*}
where $E(\cdot)$ is the encryption primitive. 
\end{definition}
The security parameter $n$ is given as input in the form $1^n$ to $A$ and $A'$ to ensure they can run in $\text{poly}(n)$-time.
%


\section{Comparison protocol}\label{app:comparison}

\setcounter{theorem}{0}
     \renewcommand{\thetheorem}{\Alph{section}.\arabic{theorem}}
\setcounter{definition}{0}
     \renewcommand{\thedefinition}{\Alph{section}.\arabic{definition}}
\setcounter{lemma}{0}
     \renewcommand{\thelemma}{\Alph{section}.\arabic{lemma}}
\setcounter{proposition}{0}
     \renewcommand{\theproposition}{\Alph{section}.\arabic{proposition}}
\setcounter{remark}{0}
     \renewcommand{\theremark}{\Alph{section}.\arabic{remark}}     
\setcounter{corollary}{0}
     \renewcommand{\thecorollary}{\Alph{section}.\arabic{corollary}}  
Consider a \mbox{two-party} computation problem under an encryption scheme that does not support comparison between encrypted data. A number of secure comparison protocols on private inputs owned by two parties have been proposed in the literature~\cite{DGK07,Garay07,Kolesnikov09improved,Lipmaa13} etc., with a survey of the state of the art given in~\cite{Couteau16}. Most of the comparison protocols have linear complexity in the size of the inputs, since the comparison is done bitwise. Out of these, the DGK protocol proposed in~\cite{DGK07} with the correction in~\cite{DGK09correction} remains one of the most computationally efficient protocols, and~\cite{Kolesnikov09improved} is the most efficient for large-scale protocols.~\cite{Lipmaa13} proposes a comparison that is sublinear in the number of invocation of the cryptographic primitive, but has greater communication complexity and is only competitive for large inputs, due to the constants involved. Depending on the specific problem, some variants might be better than others. 
Since this comparison protocol is used as a block, it can be easily replaced in Protocol~\ref{alg:comp_DGKV} once more efficient protocol are proposed.

In~\cite{DGK07,DGK09correction}, Damg{\aa}rd, Geisler and Kr{\o}igaard describe a protocol for secure comparison protocol and towards that functionality, they propose the DGK additively homomorphic encryption scheme with the property that it is efficient to determine if the value zero is encrypted, which is useful for comparisons and when working with bits. The authors also prove the semantic security of the DGK cryptosystem under the hardness of factoring assumption.

The plaintext space for DGK is $\mathbb Z_u$, where $u$ is a small prime divisor of $p-1$ and $q-1$, where $p,q$ are large, same size prime numbers, and $N=pq$. Parameters $v_p$ and $v_q$ are $t$-bit prime divisors of $p-1$ and $q-1$. The numbers $g$ and $h$ are elements of $\mathbb Z_N^\ast$ of order $uv_pv_q$ and $v_pv_q$. 
The DGK encryption scheme has public key $pk_{DGK} = (N,g,h,u)_{DGK}$ and $sk_{DGK} = (p,q,v_p,v_q)_{DGK}$. 
For a plaintext $x\in\mathbb Z_u$, a DGK encryption is:
\[[a]=g^ah^r\bmod N,\]
with $r$ a random $2t$-bits integer, which makes DGK a probabilistic encryption scheme. We use this encryption scheme with the purpose of encrypting and decrypting bits that represent comparison results, therefore, for decryption, we only need to check if the encrypted value is 0. It is enough to verify:
\[[a]^{v_pv_q}\bmod p = 1\]
to see if $a=0$. Since DGK is an additively homomorphic scheme, the following holds, where $D(\cdot)$ is the DGK decryption primitive: 
\begin{equation} 
a+b = D([a]\cdot [b] \bmod N),~~-a =D([a]^{-1}\bmod N).
\end{equation}
We will use again $\oplus$ and $\otimes$ to abstract the addition between and encrypted values, respectively, the multiplication between a plaintext and an encrypted values.

Protocol~\ref{alg:comp_DGKV} calls the DGK protocol, presented in~\cite[Protocol 3]{Veugen12}. 
The two parties have each their own private input $\alpha$, respectively $\beta$. Using the binary representations of $\alpha$ and $\beta$, the two parties exchange $l$ blinded values such that each of the parties will obtain a bit $\delta_{\mc C}\in\{0,1\}$ and $\delta_{\mc T}\in\{0,1\}$ that satisfy the following relation: $\delta_{\mc C}\veebar\delta_{\mc T} = (\alpha\leq \beta)$ after executing Protocol~\ref{alg:plain_DGKV}, where $\veebar$ denotes the exclusive or operation. Then, in Protocol~\ref{alg:comp_DGKV} line 8, $\mc C$ obtains the encrypted bit $[[t']]$ that satisfies $(t'=1)\Leftrightarrow (\beta<\alpha)$. 

To guarantee the security of the DGK scheme, we choose a key-size that makes factoring hard and set the randomization in the encryption primitive to be of length greater than $2t$ bits, for $t=160$. See more details in~\cite{DGK07,DGK09correction}.

\begin{protocol}[Protocol for secure \mbox{two-party} comparison with plaintext inputs using DGK \cite{DGK07,Veugen12}]
\label{alg:plain_DGKV}
\begin{algorithmic}[1]
\small
 \Require{$\mc C$: $\alpha$; $\mc T$: $\beta,sk\target,sk_{DGK}$}
 \Ensure{$\mc C$: $ \delta\cloud$; $\mc T$: $\delta\target$ such that $\delta\cloud\oplus\delta\target = (\alpha\leq\beta)$}
\State $\mc T$: send the encrypted bits $[\beta_i]$, $0\leq i <l$ to $\mc C$
\For{each $0\leq i <l$}
\State $\mc C$: $[\alpha_i\veebar \beta_i]\leftarrow \begin{cases} [\beta_i] &\text{ if } \alpha_i = 0\\ [1]\oplus(-1)\otimes[\beta_i] &\text{ otherwise} \end{cases}$
\EndFor
\State $\mc C$: Choose uniformly random bit $\delta\cloud\in\{0,1\}$
\State $\mc C$: Compute set $\mc L = \{i| 0\leq i <l \text{ and } \alpha_i = \delta\cloud \}$
\For{each $i\in\mc L$}
\State $\mc C$: compute $[c_i]\leftarrow \sum_{j=i+1}^l [\alpha_j\veebar\beta_j]$ \Comment encrypted sum
\State $\mc C$: $[c_i]\leftarrow \begin{cases}  [1]\oplus[c_i]\oplus(-1)\otimes[\beta_i] &\text{ if } \delta\cloud = 0\\ [1]\oplus[c_i] &\text{ otherwise} \end{cases}$
\EndFor
\State $\mc C$: generate uniformly random non-zero values $r_i$ of $2t$ bits, $0\leq i <l$
\For{each $0\leq i <l$}
\State $\mc C$: $[c_i]\leftarrow \begin{cases} r_i\otimes[c_i] & \text{ if } i\in\mc L \\  [r_i] & \text{ otherwise} \end{cases}$
\EndFor
\State $\mc C$: send the values $[c_i]$ in random order to $\mc T$
\State $\mc T$: if at least one of the values $c_i$ is decrypted to zero, set $\delta\target \leftarrow 1$, otherwise set $\delta\target\leftarrow 0$
\end{algorithmic}
\end{protocol}


\section{Proof of main results}\label{app:proof}

\allowdisplaybreaks
\setcounter{theorem}{0}
     \renewcommand{\thetheorem}{\Alph{section}.\arabic{theorem}}
\setcounter{definition}{0}
     \renewcommand{\thedefinition}{\Alph{section}.\arabic{definition}}
\setcounter{lemma}{0}
     \renewcommand{\thelemma}{\Alph{section}.\arabic{lemma}}
\setcounter{proposition}{0}
     \renewcommand{\theproposition}{\Alph{section}.\arabic{proposition}}
\setcounter{remark}{0}
     \renewcommand{\theremark}{\Alph{section}.\arabic{remark}}     
\setcounter{corollary}{0}
     \renewcommand{\thecorollary}{\Alph{section}.\arabic{corollary}}

\subsection{Proof of Theorem~\ref{thm:main_alg}}\label{subsec:proof_thm1}
For simplicity of the exposition, we avoid mentioning the public data (public keys $pk\target, pk_{DGK}$, number of iterations $K$ and number of bits $l$) in the views, since they are public. 
In what follows, we will successively discuss the views of each type of party participating in the protocol: agents, cloud and target node. As mentioned in Definitions~\ref{def:view} and~\ref{def:coalition_view}, the views of the parties during the execution of Protocol~\ref{alg:main_alg} are constructed on all the inputs of the parties involved and symbolize the real values the parties get in the execution of the protocol. We will denote by $\mc I$ the inputs of all the parties:
\[\mc I = \{b\agent,c\agent,A\cloud,Q\cloud,sk\target,sk_{DGK}\}.\]
Furthermore, in order to construct a simulator that simulates the actions of a party, we need to feed into it the inputs and outputs of that respective party. 

\subsubsection{Simulator for one agent $\mc A_i$}
Agent $\mc A_i$, for every $i=1,\ldots,p$, has the following inputs $I_{\mc A_i} = (\{b_j\}_{j=1,\ldots,m_i},$ $\{c_j\}_{j=1,\ldots,n_i})$. To avoid cluttering, we will drop the subscripts $j=1,\ldots,m_i$ and ${j=1,\ldots,n_i}$. Then agent $\mc A_i$ has the following view: 
\[V_{\mc A_i}(\mc I):= (b_j, c_j, [[b_j]], [[c_j]],\text{coins}), \]
where coins represent the random values used for encryption.

The agents are only online to send their encrypted data to the cloud, and they do not have any role and output afterwards. Hence, a simulator $S_{\mc A_i}$ would simply generate the random values necessary to encrypt its inputs as specified by the protocol and output the view obtained such:
\[S_{\mc A_i}:= (b_j, c_j, \widetilde{[[b_j]]}, \widetilde{[[c_j]]}, \widetilde{\text{coins}}), \]

where by $\widetilde{(\cdot)}$ we denote the quantities obtained by the simulator, which are different than the quantities of the agents, but follow the same distribution. Hence, it is trivial to see that the protocol is secure in the semi-honest model from the point of view of the interaction with the agents.

Next, we want to prove the privacy of the protocol from the point of view of the interactions with the cloud and the target node. We will construct a sequence of algorithms such that we obtain that the view of the real parties after the execution of $K$ iterations is the same as the view of simulators that simply execute $K$ iterations with random exchanged messages.

\subsubsection{Simulator for the cloud $\mc C$}\label{subsubsec:cloud}
For the simplicity of the exposition, we will treat $\mu$, $b\agent$ and $c\agent$ as scalars. The same steps are repeated for every element in the vectors. For clarity, we will form the view of the cloud in steps, using pointers to the lines in Protocol~\ref{alg:main_alg}. 
The view of the cloud during the execution of lines 5-6 is:
\begin{equation}\label{eq:view_cloud_-1}
V\cloud^{-1}(\mc I) = \big ( A\cloud,Q\cloud,\eta,[[b\agent]],[[c\agent]],[[\mu_0]],\text{coins}\big)=:I\cloud^{-1},
\end{equation}
where coins represent the random values generated for the Paillier encryption. Furthermore, we construct the view of the cloud at iteration $k=0,\ldots,K-1$ during the execution of an instance of Protocol~\ref{alg:iteration}, which will be constructed on the inputs of all parties: the inputs $\mc I$ and the data the parties had at iteration $k-1$. We denote the view of the cloud at iteration $k-1$ by $I\cloud^{k-1}$ which, along with $\mc I$ and the view of the target node at iteration $k-1$, denoted by $I\target^{k-1}$~\eqref{eq:view_target_k}, will be what the view is constructed on at iteration $k$. 
\begin{align*}
&\bar{\mc I}^{k-1} := \mc I \cup I\cloud^{k-1} \cup I\target^{k-1},\numberthis \label{eq:inputs_k}\\
&I\cloud^{k}:=V\cloud^k(\bar{\mc I}^{k-1}) = \big(I\cloud^{k-1},[[\mu_k]],[[\bar\mu_k]],\underbrace{\pi_k,\text{coins}_{1k}}_{Protocol~\ref{alg:rand_step}}, \numberthis \label{eq:view_cloud_k}\\
&\underbrace{\rho_k,[[t_k]],\text{m}^{\text{comp}_k},\text{coins}_{2k}}_{Protocol~\ref{alg:comp_DGKV}}, \underbrace{r_k,s_k,[[v_k]],[[\mu_{k+1}]],\text{coins}_{3k}}_{Protocol~\ref{alg:update_step}}\big),
\end{align*}
where $\text{m}^{\text{comp}_k}$ are messages exchanged in the comparison protocol, and $\text{coins}_{jk}$, for $j=1,2,3$ are the random numbers generated in Protocol~$j$. Finally, the view of the cloud after the execution of line 11 in Protocol~\ref{alg:main_alg} is:
\begin{equation}\label{eq:view_cloud_K}
V\cloud^K(\bar{\mc I}^{K-1}) = \big (I^{K-1}\cloud,[[x^\ast]]).
\end{equation}

Therefore, the view of the cloud during the whole execution of Protocol~\ref{alg:main_alg} is:
\[V\cloud(\mc I) := V\cloud^K(\bar{\mc I}^{K-1}).\]

We first construct a simulator on the inputs $I\cloud = \{A\cloud, Q\cloud\}$ that mimics $V^{-1}\cloud(\mc I)$ in~\eqref{eq:view_cloud_-1}:
\begin{itemize}[wide, labelwidth=!, labelindent=0pt]
	\item Generate $n+m$ random numbers of $l$ bits $\widetilde {b\agent},\widetilde{c\agent}$;
	\item Generate a random positive initial value $\widetilde{\mu_0}$;
	\item Generate $n+m+1$ uniformly random numbers for the Paillier encryption and denote them $\widetilde{\text{coins}}$;
	\item Compute $[[\widetilde {b\agent}]],[[\widetilde{c\agent}]],[[\widetilde{\mu_0}]]$;
	\item Compute $\eta$ following line 6;
	\item Output $\widetilde I\cloud^{-1} := S^{-1}\cloud(I\cloud) = \big (A\cloud, Q\cloud,[[\widetilde{b\agent}]],[[\widetilde{c\agent}]],\eta,[[\widetilde{\mu_0}]],$ $\widetilde{\text{coins}}\big)$.
\end{itemize}

Since Protocol~\ref{alg:iteration} is secure in the semi-honest model (Proposition~\ref{prop:iteration}), we know that there exists a probabilistic polynomial-time (ppt) simulator for the functionality of Protocol~\ref{alg:iteration} on inputs $(A\cloud,Q\cloud,[[b\agent]],[[c\agent]],\eta,[[\mu_k]])$ and output $[[\mu_{k+1}]]$. However, we need to show that we can simulate the functionality of consecutive calls of Protocol~\ref{alg:iteration}, or, equivalently, on one call of Protocol~\ref{alg:iteration} but on the augmented input that contains the data of the cloud in the previous iterations. Call such a simulator $S^k\cloud$, that on the input $I^{k-1}\cloud$ mimics $V^k\cloud(\bar{\mc I}^{k-1})$ in~\eqref{eq:view_cloud_k}, for $k=0,\ldots,K-1$:
\begin{itemize}[wide, labelwidth=!, labelindent=0pt]
	\item Compute $[[\nabla g(\mu_k)]]$ and $[[\bar \mu_k]]$ as in lines 1-2 of Protocol~\ref{alg:iteration} from $[[\mu_k]],[[b\agent]],[[c\agent]]$ which are included in $I^{k-1}\cloud$;
	\item Generate a random permutation $\widetilde{\pi_k}$ and apply it on $([[0]],[[\bar \mu_k]])$ as in Protocol~\ref{alg:rand_step};
	\item Follow Protocol~\ref{alg:comp_DGKV} and replace the incoming messages by DGK encryptions of appropriately sized random values to obtain $\widetilde{\rho_k},\widetilde{\text{m}^{\text{comp}_k}}$;
	\item Generate random bit $\widetilde{t_k}$ and its encryption $[[\widetilde{t_k}]]$;
	\item Generate random values $\widetilde{r_k}$ and $\widetilde{s_k}$ as in line 1 in Protocol~\ref{alg:update_step} and their encryptions $[[\widetilde{r_k}]],[[\widetilde{s_k}]]$;
	\item Obtain $[[\widetilde{v_k}]]$ by choosing between the elements of $\widetilde{\pi_k}([[0]],[[\bar\mu_k]])+(\widetilde{r_k},\widetilde{s_k})$ according to the generated $\widetilde{t_k}$;
	\item Compute $[[\widetilde{\mu_{k+1}}]]$ as in line 8 of Protocol~\ref{alg:update_step};
	\item Denote the rest of the random values used for encryption and blinding by $\widetilde{\text{coins}_k}$;
	\item Output $\widetilde I\cloud^k :=S^k\cloud(I^{k-1}\cloud)=(I^{k-1}\cloud,[[\bar \mu_k]],[[\widetilde{\pi_k}]],[[\widetilde{z_k}]],$ $\widetilde{\text{m}^{\text{comp}_k}},[[\widetilde{t_k}]],[[\widetilde{r_k}]],[[\widetilde{s_k}]],[[\widetilde{v_k}]],[[\widetilde{\mu_{k+1}}]],\widetilde{\text{coins}_k})$.
\end{itemize}

Finally, a trivial simulator $\widetilde I\cloud^K:= S^K\cloud(I^{K-1}\cloud)$ for $V^K\cloud(\bar{\mc I}^{K-1})$ is obtained by simply performing line 11 on the inputs.

\begin{proposition}\label{prop:sim_cloud}
	$S\cloud^{k}(I\cloud^{k-1})\stackrel{c}{\equiv}V^{k}\cloud(\bar{\mc I}^{k-1})$, for \mbox{$k=-1,\ldots,K$}, where $I\cloud^{k-2}:=I\cloud$.
\end{proposition}
\begin{proof}
For $k=-1$, due to the fact that the coins and the initial iterate are generated by the simulator via the same distributions as specified in the protocol, and due to the semantic security of the Paillier encryption, which guarantees indistinguishability of encryptions, we obtain that $V^{-1}\cloud(\mc I)\stackrel{c}{\equiv}S^{-1}\cloud(I\cloud)$.

For $0\leq k\leq K-1$, $(\text{coins}_{1k},\text{coins}_{2k},\text{coins}_{3k})\stackrel{s}{\equiv}\widetilde{\text{coins}_k}$ because they are generated from the same distributions. Similarly, the distributions of $(\widetilde{\pi_k},\widetilde{\pi_k}([[0]],[[\bar \mu_k]]))$ and $(\pi_k,\pi_k([[0]],[[\bar \mu_k]]))$ are the same and $\widetilde{\pi_k},\pi_k$ are independent of the other parameters. The quantities from the DGK protocol $\widetilde{\rho_k}\stackrel{s}{\equiv}\rho_k$ and $\widetilde{\text{m}^{\text{comp}_k}}\stackrel{c}{\equiv}\text{m}^{\text{comp}_k}$ either due to the semantic security of the DGK protocol, e.g., $[\widetilde{t'}],[\widetilde{\delta\target}]$ or due to having the same distributions, e.g., $\widetilde{\rho_k}, \widetilde{\delta\cloud},\widetilde{\alpha}$. The values in the update step $\widetilde{r_k},\widetilde{s_k}$ are sampled from the same distribution as $r_k,s_k$, and, finally, $[[\widetilde{t_k}]],[[\widetilde v_k]],[[\widetilde{\mu_{k+1}}]]\stackrel{c}{\equiv}[[t_k]],[[v_k]],[[\mu_{k+1}]]$ due to the semantic security of the Paillier encryption. Thus, $S^{k}(I\cloud^{k-1})\stackrel{c}{\equiv}V^{k}\cloud(\bar{\mc I}^{k-1})$.

For $k=K$, $S\cloud^{K}(I\cloud^{K-1})\stackrel{c}{\equiv}V^{K}\cloud(\bar{\mc I}^{K-1})$ follows from the fact that the simulator simply executes the last part of the protocol and from the Paillier scheme's semantic security.
\end{proof}

Hence, we obtained that $I^k\cloud \stackrel{c}{\equiv} \widetilde I^k\cloud$, for $k=-1,\ldots,K$. 
The essence of the proof of Proposition~\ref{prop:sim_cloud} is that all the messages the cloud receives are encrypted. Then, thanks to the semantic security of the Paillier and DGK schemes, the extra information included in $\bar{\mc I}^{k-1}$ from the previous iterations cannot be used to extract other information about the values at iteration $k$. From this, we also have the next corollary:

\begin{corollary}\label{corr:consec_steps}
	$S^{k}\cloud(\widetilde I^{k-1}\cloud)\stackrel{c}{\equiv}S^{k}\cloud(I\cloud^{k-1})$, equivalent to $S^{k+1}\cloud(S^{k}\cloud( I\cloud^{k-1}))\stackrel{c}{\equiv}S^{k+1}\cloud(V^{k}\cloud(\bar{\mc I}^{k-1}))$, for any $k=0,\ldots,K-1$.
\end{corollary}

Finally, we construct a simulator $S\cloud(I\cloud)$ for the execution of Protocol~\ref{alg:main_alg} and we will show that its view will be computationally indistinguishable from $V\cloud(\mc I)$. To this end, we define the following sequence of views -- obtained as hybrids between the real views and the views of the simulators:
\begin{align*}
	V\cloud(\mc I)=H_{-1}(\mc I) &= V^K\cloud (\bar {\mc I}^{K-1})\\
	H_0(\mc I) &= S^K\cloud \circ V\cloud^{K-1}(\bar {\mc I}^{K-2})\\
	H_1(\mc I) &= S^K\cloud \circ S\cloud^{K-1} \circ V\cloud^{K-2}(\bar {\mc I}^{K-3})\\
	&~\vdots\\
	H_{K}(\mc I) &= S^K\cloud\circ S^{K-1}\cloud\circ \ldots\circ S\cloud^0\circ V^{-1}\cloud(\mc I)\\
	S\cloud(I\cloud)=H_{K+1}(I\cloud) &= S^K\cloud\circ S^{K-1}\cloud\circ \ldots\circ S\cloud^0\circ S^{-1}\cloud(I\cloud).
\end{align*}

By transitivity, $H_{-1}$ and $H_{K+1}$ are computationally indistinguishable if:
\[
H_{-1} \stackrel{c}{\equiv} H_0\stackrel{c}{\equiv} \ldots  \stackrel{c}{\equiv} H_k \stackrel{c}{\equiv} H_{k+1} \stackrel{c}{\equiv} H_{k+2} \stackrel{c}{\equiv} \ldots \stackrel{c}{\equiv} H_{K+1}. 
\]
This result follows from induction on Corollary~\ref{corr:consec_steps}. 
In conclusion, we obtain that $S\cloud(I\cloud)\stackrel{c}{\equiv} V\cloud(\mc I)$, which verifies that Protocol~\ref{alg:main_alg} achieves privacy with respect to the cloud.

\subsubsection{Simulator for the target node $\mc T$}\label{subsubsec:target}
We proceed with the same steps in order to show that the consecutive $K$ iterations form a protocol that is secure in the semi-honest model from the point of view of the target node. We will use the secrecy of the one-time pad variant used for blinding. The symmetric encryption used in this paper, as discussed in Section~\ref{subsec:symmetric}, to which we refer as blinding, guarantees that a value of $l$ bits additively blinded by a random value of $l+\lambda$ bits is statistically indistinguishable from a random value of $l+\lambda+1$ bits.

The inputs and output of the target node in Protocol~\ref{alg:main_alg}~are $I\target = (sk\target,sk_{DGK},x^\ast)$. 
As in~\eqref{eq:inputs_k}, $\bar{\mc I}^{k-1}$ represents the~inputs of all the parties at iteration $k$, with $\bar{\mc I}^{-1} = \mc I$. Then, the view of the target node during iteration $k=0,\ldots,K-1$ is:
\begin{align*}
I\target^{k}:=&V\target^k(\bar{\mc I}^{k-1}) = (sk\target,sk_{DGK},\underbrace{z_k,t_k,\text{m}^{\text{comp}_k},\text{coins}_{2k}}_{Protocol~\ref{alg:comp_DGKV}},\\
& \underbrace{\bar a_k,\bar b_k, v_k,\text{coins}_{3k}}_{Protocol~\ref{alg:update_step}}). \numberthis\label{eq:view_target_k}
\end{align*}

The view of the target node during the last step of the protocol is:
\begin{equation}V\target^K(\bar{\mc I}^{K-1}) := (I\target^{K-1},[[x^\ast]]).\end{equation}

As before, the view of the target during the execution of Protocol~\ref{alg:main_alg} is:
\[V\target(\mc I) := V\target^K(\bar{\mc I}^{K-1}).
\]

In order to be able to construct a simulator with indistinguishable view from the view of the target node, we need to show that the target node is not capable of inferring new relevant information about the private data (other than what can be inferred solely from its inputs and outputs) although it knows the optimal solution $x^\ast$ and has access to the messages from multiple iterations. 
 
Apart from the last message, which is the encryption of the optimization solution $[[x^\ast]]$, and the comparison results $t_k$, all the values the target node receives are blinded, with different sufficiently large values at each iteration. 
The target node knows that $Q\cloud x^\ast = A\cloud^\intercal \mu_K - c\agent$ and that $\mu_K = v_K - \bar r_K$, for some random value $\bar r_K$. However, from these two equations, although it has access to $x^\ast$ and $v_K$, the target node cannot infer any new information about the private data of the agents or about $\mu_K$, even when $Q\cloud,A\cloud$ are public, thanks to the large enough randomization.

Let $a_k,b_k = \pi_k(0,\bar\mu_{k})$, where $\pi_k$ is a random permutation. From Protocol~\ref{alg:comp_DGKV}, the target receives $z_k$ which is the additively blinded value of $b_k-a_k+2^l$, and other blinded values, denoted in~\eqref{eq:view_target_k} as $\text{m}^{\text{comp}_k}$. Hence, provided the blinding noises are refreshed at every iteration, the target node cannot infer any information, as follows from Section~\ref{subsec:symmetric}. At the end of the protocol, it receives the bit $t_k$ which is 1 if $a_k\leq b_k$ and 0 otherwise. However, due to the uniform randomization between the order of $\bar\mu_k$ and $0$ at every iteration for assigning the values of $a_k$ and $b_k$, described in Protocol~\ref{alg:rand_step}, the target node cannot identify the sign of $\bar \mu_k$, and by induction, the sign of $\bar \mu_{K-1}$ and magnitude of $\mu_K$. This means that having access to $x^\ast = x_K$ does not bring more information about the blinded values from the intermediary steps. 

We now investigate the relation between the messages from consecutive iterations. From the update protocol~\ref{alg:update_step}, we know: 
\begin{align*}
 v_k  &= (a_k+r_k)(1-t_k) + (b_k+s_k)t_k \\ 
 \mu_{k+1} &= v_k - r_k(1-t_k) - s_kt_k = a_k(1-t_k) + b_kt_k,
 \end{align*}
where the variables are vectors. The target node knows the values of $v_k$ and $t_k$, but $\pi_k,\bar\mu_k,r_k,s_k$ are unknown to it. 
Furthermore, let $\pi_{k+1}$ be the permutation applied by the cloud at step $k+1$ in Protocol~\ref{alg:rand_step}, unknown to the target node. Take for example the case when $t_k = 0$ and $t_{k+1} = 0$. Let $\tilde Q = \mathbf I - \eta A\cloud Q\cloud^{-1} A\cloud^\intercal$ and $\mathbf E = [\mathbf I~ \mathbf 0]$. Then:
\begin{align*}
v_{k+1} &= a_{k+1} + r_{k+1}= \mathbf E \pi_{k+1} (0, \tilde Q(v_k - r_k)  -\eta A\cloud Q\cloud^{-1}c\agent -\\ &- \eta b\agent ) + r_{k+1}.
\end{align*}
The above equation shows the target node cannot construct $v_{k+1}$ from $v_{k}$. Similar equations arise when considering the other values for $t_k$ and $t_{k+1}$. This guarantees that an integer $\tilde v_k^i$ obtained by selecting uniformly at random from $(2^{l+\lambda},2^{l+\lambda+1})\cap\mathbb Z_N$ will have the distribution statistically indistinguishable from $v_k^i$. Moreover:
\begin{align*}
&\mu_{k+2} = \max\{0,\tilde Q\mu_{k+1} -\eta A\cloud Q\cloud^{-1}c\agent - \eta b\agent\} \\
	&= \max\{0, \tilde Q(v_k - r_k(1-t_k) - s_kt_k) -\eta A\cloud Q\cloud^{-1}c\agent - \eta b\agent\} \\
	&= v_{k+1} - r_{k+1}(1-t_{k+1}) - s_{k+1}t_{k+1} .
\end{align*}
Since the blinding noise is different at each iteration and uniformly sampled, the target node cannot retrieve $\mu_{k+2}$ from $v_k$ and $v_{k+1}$. In short, if the target node receives some random values $\widetilde{\bar a_k},\widetilde{\bar b_k}\in (2^{l+\lambda},2^{l+\lambda+1})\cap \mathbb Z_N$ instead of $a_k + r_k$ and $b_k+s_k$ respectively, it would not be able to distinguish the difference. Similar arguments hold for the blinded messages $c_i$ from the comparison protocol~\ref{alg:plain_DGKV}.

Hence, by processing multiple iterations, the target node can only obtain functions of the private data that involve at least one large enough random value, which does not break privacy.

We now build a simulator $S\target$ that applies the steps of the protocol on randomly generated values. As before, 
since Protocol~\ref{alg:iteration} is secure in the semi-honest model (Proposition~\ref{prop:iteration}), we know that there exists a ppt simulator for the functionality of Protocol~\ref{alg:iteration} on inputs $I\target^{-1} = \{sk\target,sk_{DGK}\}$. However, we need to show that we can simulate the functionality of consecutive calls of Protocol~\ref{alg:iteration}, or, equivalently, on one call of Protocol~\ref{alg:iteration} but on inputs $(I\target^k, x^\ast)$. Call such a simulator $S^k\target$, that on the inputs $(I\target^k, x^\ast)$ should output a view that is statistically indistinguishable from $V^k\target(\bar{\mc I}^{k-1})$ in~\eqref{eq:view_target_k}, for $k=0,\ldots,K-1$. We already showed that although the target node has access to the output $x^\ast$ and to blinded messages from all the iterations of Protocol~\ref{alg:main_alg}, it cannot extract information from them or correlate the messages to the iteration they arise from, so $x^\ast$ will only be relevant in the last simulator.
\begin{itemize}[wide, labelwidth=!, labelindent=0pt]
	\item Generate a $\lambda+l$-length random integer $\widetilde{\rho_k}$ and add $2^l$ and obtain $\widetilde{z_k}$;
	\item Generate a random bit $\widetilde{t_k}$;
	\item Choose a random bit $\widetilde{\delta\target}$. If it is 0, then generate $l$ non-zero values of $2t$ bits, otherwise generate $l-1$ non-zero random values and one 0 value. Those will be the $\widetilde{\text{m}^{\text{comp}_k}}$ (see Protocol~\ref{alg:plain_DGKV});
	\item Generate random integers of length $l+ \lambda+1$ $\widetilde{\bar a_k}$ and $\widetilde{\bar b_k}$;
	\item Compute $\widetilde{v_k}$ according to $\widetilde{t_k}$;
	\item Denote all Paillier and DGK generated coins by $\widetilde{\text{coins}}$;
	\item Output $\widetilde I\target^{k}:= S\target^k(I\target^{k-1},x^\ast)=(I\target^{k-1},\widetilde{z_k},\widetilde{t_k},\widetilde{\text{m}^{\text{comp}_k}},\widetilde{\bar a_k},$ $\widetilde{\bar b_k},\widetilde{v_k},$ $\widetilde{\text{coins}})$.
\end{itemize}

Finally, a trivial simulator $\widetilde I\target^{K}:=S^K\target(I^{K-1}\target,x^\ast)$ for $V^K\target(\bar{\mc I}^{K-1})$ is obtained by simply generating an encryption of $x^\ast$ and outputting: $(I^{K-1}\target,\widetilde{[[x^\ast]]})$.

\begin{proposition}\label{prop:sim_target}
	$S\target^{k}(I\target^{k-1},x^\ast)\stackrel{c}{\equiv}V^{k}(\bar{\mc I}^{k-1})$, for \mbox{$k=0,\ldots,K$}.
\end{proposition}
\begin{proof}

For $0\leq k\leq K-1$, $(\text{coins}_{2k},\text{coins}_{3k})\stackrel{s}{\equiv}\widetilde{\text{coins}_k}$ because they are generated from the same distributions. Similarly, $\widetilde{z_k}\stackrel{s}{\equiv}z_k$ because $\widetilde{\rho_k}\stackrel{s}{\equiv}b_k-a_k+\rho_k$. From the discussion above, the same holds for $\widetilde{\text{m}^{\text{comp}_k}}$ and their counterparts $\text{m}^{\text{comp}_k}$. Furthermore, $(\widetilde{t_k},\widetilde{\bar a_k},\widetilde{\bar b_k},\widetilde{v_k})$ are statistically indistinguishable from $(t_k,\bar a_k, \bar b_k,v_k)$ due to the way they are generated, and $\widetilde{v_k}$ being consistent with $\widetilde{t_k}$. 
Thus, $S^{k}(I\target^{k-1},x^\ast)\stackrel{c}{\equiv}V^{k}(I\target^{k-1}\cup \mc I)$.

For $k=K$, $S^{K}(I\target^{K-1},x^\ast)\stackrel{c}{\equiv}V^{K}(\bar{\mc I}^{K-1})$ trivially follows from the fact that both $[[x^\ast]]$ and $\widetilde{[[x^\ast]]}$ are decrypted in $x^\ast$.
\end{proof}

Hence, we obtained that $\widetilde I\target ^k \stackrel{c}{\equiv} I\target^k$, for $k=0,\ldots,K$.
The essence of the proof of Proposition~\ref{prop:sim_target} is that all the messages the target receives are blinded with large enough random noise and the comparison bits are randomized. This results in the messages being statistically indistinguishable from random values of the same length, which means that the extra information included in $(I^{k-1}\target,x^\ast)$ from the previous iterations cannot be used to extract other information about the current values at iteration~$k$. The next corollary then follows from Proposition~\ref{prop:sim_target}:

\begin{corollary}\label{corr:consec_steps_target}
	$S\target^k(\widetilde I\target^{k-1},x^\ast)\stackrel{c}{\equiv} S\target^k(I\target^{k-1},x^\ast)$, equivalent to $S\target^{k+1}(S\target^{k}(I\target^{k-1},x^\ast),x^\ast)\stackrel{c}{\equiv}S^{k+1}\target(V^{k}(\bar{\mc I}^{k-1}),x^\ast)$, for any $k=1,\ldots,K-1$.
\end{corollary}

Finally, we construct a simulator $S\target(I\target)$ for the execution of Protocol~\ref{alg:main_alg} and we show that its view will be statistically indistinguishable from $V\target(\mc I)$. To this end, we define the following sequence, from which we drop the input $x^\ast$ to the simulators to not burden the notation:
\begin{align*}
	V\target(\mc I)=H_0(\mc I) &= V^K\target(\bar{\mc I}^{K-1})\\
	H_1(\mc I,x^\ast) &= S^K\target\circ V^{K-1}\target(\bar{\mc I}^{K-2})\\
	H_2(\mc I,x^\ast) &= S^K\target\circ S^{K-1}\target \circ V^{K-2}\target(\bar{\mc I}^{K-3})\\
	&~\vdots\\
	H_{K}(\mc I,x^\ast) &= S^K\target\circ S^{K-1}\target\circ \ldots\circ S^1\target\circ V^0\target(\mc I)\\
	S\target(I\target)=H_{K+1}(I\target) &= S^K\target\circ S^{K-1}\target\circ \ldots\circ S^1\target\circ S^0\target(I\target)
\end{align*}

By transitivity, $H_0$ and $H_{K+1}$ are statistically indistinguishable if:
\[
H_0 \stackrel{c}{\equiv} H_1\stackrel{c}{\equiv} \ldots  \stackrel{c}{\equiv} H_k \stackrel{c}{\equiv} H_{k+1} \stackrel{c}{\equiv} H_{k+2} \stackrel{c}{\equiv} \ldots \stackrel{c}{\equiv} H_{K+1}. 
\]
The result follows from induction on Corollary~\ref{corr:consec_steps_target}. 
In conclusion, we obtain that $S\target(I\target)\stackrel{c}{\equiv} V\target(\mc I)$ which verifies that Protocol~\ref{alg:main_alg} achieves privacy with respect to the target node.

The proof of Theorem~\ref{thm:main_alg} is now complete. \hfill \qed

\subsection{Proof of Theorem~\ref{thm:main_alg_coal}}\label{subsec:proof_thm2}
We will now show that any coalition consistent with the setup of Propositions~\ref{prop:assum_1},~\ref{prop:assum_2} and with the assumption that the cloud and target node do not collude does not gain any new information about the private data of the honest agents, other than what can be inferred solely from the inputs and outputs of the coalition. As mentioned in Section~\ref{sec:security_proof}, the agents in a coalition only add their inputs to the view of the coalition, but do not have any messages in the protocol.

\subsubsection{Simulator for the cloud $\mc C$ and $\bar p$ agents $\mc A_i$}
	Consider the coalition between a set of agents $\mc A_{i=1,\ldots,\bar p}$ and the cloud $\mc C$, which has the inputs $(\{b_i\}_{i=1,\ldots,\bar m},\{c_i\}_{i=1,\ldots,\bar n},A\cloud,Q\cloud)$ and no output from the execution of Protocol~\ref{alg:main_alg}. Since $\bar p < p$, 
the coalition is not able to compute $\mu_1$ by itself, and the semantic security of the Paillier cryptosystem is again enough to ensure the privacy of the rest of the sensitive data. A simulator for this coalition can be constructed following the same steps in Section~\ref{subsubsec:cloud} on the augmented inputs defined above, from the fact that every value the cloud receives is encrypted by a semantically secure cryptosystem.

\subsubsection{Simulator for the target node $\mc T$ and $\bar p$ agents $\mc A_i$}
	Consider the coalition between a set of agents $\mc A_{i=1,\ldots,\bar p}$ and the target node $\mc T$, which has the inputs $(\{b_i\}_{i=1,\ldots,\bar m},\{c_i\}_{i=1,\ldots,\bar n},sk\target,sk_{DGK})$ and output $(x^\ast)$ from the execution of Protocol~\ref{alg:main_alg}. If both matrices $Q\cloud, A\cloud$ are public, if there exists $i\in\{1,\ldots,\bar p\}$ such that $a_i^\intercal x^\ast < b_i$, the coalitions finds out $\mu^\ast_i=0$, which comes from public knowledge in the KKT conditions~\eqref{eq:KKT4}. From this, the coalition is able to find some coins of the cloud: $r_K$ and $s_K$ associated to element $i$. However, these values are independent from the rest of the parameters and do not reveal any information about the private inputs of the parties. Apart from this, the coalition is not able to compute any private data of the honest parties from the execution of the protocol, due to the secure blinding. 
A simulator for this coalition can be build by following the same steps as described in Section~\ref{subsubsec:target} on the augmented inputs defined above.

The proof is now complete.\hfill \qed

\subsection{Sketch proof of Theorem~\ref{thm:alternative}}\label{subsec:proof_thm3}
Following similar steps as in the proof of the main protocol in Appendix~\ref{subsec:proof_thm1}, simulators for the involved parties can be constructed from their inputs and outputs. A simulator for the cloud randomly generates the messages and the indistinguishability between the views is guaranteed by the indistinguishability of the Pailler encryptions. The outputs of Protocol~\ref{alg:iteration_alternative} to the target node are different than the original Protocol~\ref{alg:iteration}. This means that the outputs for the target node in Protocol~\ref{alg:main_alg} with the iterations as in Protocol~\ref{alg:iteration_alternative} are $(x^\ast, \{(r_k)_i(\bar\mu_k)_i\}_{i=1,\ldots,m,~k=0,\ldots,K-1})$. In this case, a simulator for the target node simply outputs its inputs as messages in the view, and indistinguishability follows trivially. For the multi-party privacy, simulators for the allowed coalitions can be easily constructed by expanding the simulators for the cloud and the target node with the auxiliary inputs from the agents in the coalition.\hfill \qed

\subsection{Proof of Proposition~\ref{prop:assum_1}}\label{subsec:proof_assum_1}
The coalition has access to the following data, which is fixed: $A\cloud, Q\cloud, x^\ast, \{b_i\}_{1,\ldots,\bar m}, \{c_i\}_{i,\ldots, \bar n}$.

Proof of (1). We need to address two cases: the non-strict satisfaction of the constraints and the equality satisfaction of the constraints.

\noindent(I) Suppose there exists a solution $\{b_i\}_{\bar{m}+1, \ldots, m}$ and $\{c_i\}_{\bar{n}+1, \ldots, n}$ and $\mu$ to the KKT conditions such that $a_i^\intercal x^\ast < b_i$ for some ${\bar{m}+1} \leq i \leq m$. In particular this implies that $\mu_i = 0$. Then define $c' := c\agent$, $\mu' := \mu$ and $b'$ such that $b'_j := b_j$ for all $j \neq i$ and $b_i'$ to take any value $b_i' \geq a_i^\intercal x^\ast$. The new set of points $(b', c', \mu')$ is also a solution to the KKT conditions, by construction. 

\noindent(II) Alternatively, suppose there exists a solution $\{b_i\}_{\bar{m}+1, \ldots, m}$ and $\{c_i\}_{\bar{n}+1, \ldots, n}$ and $\mu$ to the KKT conditions such that $a_j^\intercal x^\ast = b_j$ for all $j={\bar{m}+1}, \ldots, m$. Consider there exists a vector $\delta$ that satisfies $\delta\succeq 0$ and $A_{21}^\intercal \delta = 0$. Compute $\epsilon \geq 0$ as: $\epsilon = \min \left(\frac{\mu_k}{\delta_k}\right)_{\delta_k>0, k=\bar m+1,\ldots,m}$. Then, we construct $\mu' := \mu - \epsilon \left[ \begin{array}{c} 0 \\ \delta \end{array}\right]$ that satisfies $\mu'\succeq 0$ and $\mu'_i =0$ for some ${\bar{m}+1} \leq i \leq m$ that is the argument of the above minimum. 

Furthermore, define $c' := c\agent + \epsilon\left[ \begin{array}{c} 0\\  A_{22}^\intercal \delta \end{array}\right] $ and $b'$ such that $b'_j := b_j$ for all $j \neq i$ and $b'_i $ to be any value  $b'_i > b_i$. Then $(b', c', \mu')$ is also a solution to the KKT conditions. More specifically, the complementarity slackness condition holds for all $j={\bar{m}+1}, \ldots, m$:
\begin{equation}
\begin{aligned}
&~\mu'_j ~ (a_j^\intercal x^\ast - b'_j) =&& \mu_j' (\underbrace{a_j^\intercal x^\ast - b_j}_{=0}) =  0, ~~ j \neq i\\
& \underbrace{\mu'_i}_{=0} (a_i^\intercal x^\ast - b'_i)=&& 0, ~~ j = i.
\end{aligned}
\end{equation}
Then we can check the gradient condition:
\begin{align}
\begin{split}
&Q\cloud x^\ast + A\cloud^\intercal  \mu' + c' =\\
&Q\cloud x^\ast + A\cloud^\intercal \left( \mu - \epsilon \left[ \begin{array}{c} 0 \\ \delta \end{array}\right] \right) + \left(c + \epsilon\left[ \begin{array}{c} 0\\  A_{22}^\intercal \delta \end{array}\right] \right) \\
&=Q\cloud x^\ast + A\cloud^\intercal  \mu + c\agent =  0.
\end{split}
\end{align}
Hence, $b'\neq b\agent$ satisfies the KKT conditions and the coalition cannot uniquely determine $b\agent$.

Proof of (2). Consider a solution $\{b_i\}_{\bar{m}+1, \ldots, m}$ and $\{c_i\}_{\bar{n}+1, \ldots, n}$ and $\mu$ to the KKT conditions. For some $\epsilon >0$ define  $\mu' := \mu + \epsilon \left[ \begin{array}{c} 0 \\ \delta \end{array}\right]$ and $c' := c\agent - \epsilon\left[ \begin{array}{c} 0\\  A_{22}^\intercal \delta \end{array}\right] $. Define $b'$ such that for all $j$, it holds that $b'_j = a_j^\intercal x^\ast$. Then $(b', c', \mu')$ is also a solution to the KKT conditions. Specifically, it follows that $\mu'\geq 0$. Moreover the complementarity slackness condition holds by construction of $b'$, and as before the gradient condition holds. Hence, $c'\neq c\agent$ satisfies the KKT solution, and the coalition cannot uniquely determine $c\agent$.
\hfill \qed
%

\end{document}